\documentclass[a4paper,reqno]{amsart}
\usepackage[english]{babel}
\usepackage{amsmath,amsthm,amssymb,amsfonts}
\usepackage{bbm}
\usepackage{enumitem}
\usepackage{hyperref}

\newtheorem{thm}{Theorem}[section]
\newtheorem{lem}[thm]{Lemma}
\newtheorem{prp}[thm]{Proposition}
\newtheorem{cor}[thm]{Corollary}
\newtheorem{dfn}[thm]{Definition}

\theoremstyle{definition}
\newtheorem{rem}[thm]{Remark}

\newcommand{\C}{\mathbb{C}}
\newcommand{\R}{\mathbb{R}}
\newcommand{\N}{\mathbb{N}}
\newcommand{\Z}{\mathbb{Z}}

\newcommand{\Sphere}{\mathbb{S}}          
\newcommand{\weight}{\omega}              

\newcommand{\tc}{\,:\,}

\newcommand{\defeq}{\mathrel{:=}}

\DeclareMathOperator{\tr}{tr}             
\newcommand{\id}{\mathop{\mathrm{id}}}    
\DeclareMathOperator*{\esssup}{ess\,sup}  
\DeclareMathOperator{\supp}{supp}         

\newcommand{\FE}{\mathcal{E}}             
\newcommand{\Trv}{\mathcal{T}}            
\newcommand{\opL}{\mathfrak{L}}           

\newcommand{\group}[1]{\mathrm{#1}}       

\newcommand{\dist}{\varrho}                        

\newcommand{\boxb}{\Box_b}                         

\newcommand{\lie}{\mathfrak}
\newcommand{\vecU}{\mathbf{U}}

\newcommand{\II}{\mathcal{I}}
\newcommand{\Bdd}{\mathcal{B}}

\newcommand{\chr}{\mathbbm{1}}
\newcommand{\thr}{\varsigma}

\newcommand{\sloc}{\mathrm{sloc}}

\newcommand{\Heis}{H}

\DeclareFontFamily{U}{matha}{\hyphenchar\font45}
\DeclareFontShape{U}{matha}{m}{n}{
      <5> <6> <7> <8> <9> <10> gen * matha
      <10.95> matha10 <12> <14.4> <17.28> <20.74> <24.88> matha12
      }{}
\DeclareSymbolFont{matha}{U}{matha}{m}{n}
\DeclareFontSubstitution{U}{matha}{m}{n}

\DeclareFontFamily{U}{mathx}{\hyphenchar\font45}
\DeclareFontShape{U}{mathx}{m}{n}{
      <5> <6> <7> <8> <9> <10>
      <10.95> <12> <14.4> <17.28> <20.74> <24.88>
      mathx10
      }{}
\DeclareSymbolFont{mathx}{U}{mathx}{m}{n}
\DeclareFontSubstitution{U}{mathx}{m}{n}

\DeclareMathDelimiter{\vvvert}{0}{matha}{"7E}{mathx}{"17}

\newcommand{\enref}[1]{\textup{(\ref{#1})}}

\begin{document}
\title[Joint functional calculi and a sharp multiplier theorem]{Joint functional calculi and a sharp multiplier theorem for the Kohn Laplacian on spheres}

\author[A. Martini]{Alessio Martini}
\address{School of Mathematics\\ University of Birmingham\\Edgbaston\\Birmingham B15 2TT \\ United Kingdom}
\email{a.martini@bham.ac.uk}


\begin{abstract}
Let $\boxb$ be the Kohn Laplacian acting on $(0,j)$-forms on the unit sphere in $\C^n$. In a recent paper of Casarino, Cowling, Sikora and the author, a spectral multiplier theorem of Mihlin--H\"ormander type for $\boxb$ is proved in the case $0<j<n-1$. Here we prove an analogous theorem in the exceptional cases $j=0$ and $j=n-1$, including a weak type $(1,1)$ endpoint estimate. We also show that both theorems are sharp. The proof hinges on an abstract multivariate multiplier theorem for systems of commuting operators.
\end{abstract}

\maketitle

\section{Introduction}

Let $X$ be a measure space, $\FE$ be a complex vector bundle on $X$ with a hermitian metric, and $\opL$ be a (possibly unbounded) self-adjoint operator on the space $L^2(\FE)$ of $L^2$-sections of $\FE$. By the spectral theorem, we can write
\begin{equation}\label{eq:spres}
\opL = \int_\R \lambda \, dE_\opL(\lambda)
\end{equation}
for some projection-valued measure $E_\opL$, called the spectral resolution of $\opL$. A functional calculus for $\opL$ is then defined via spectral integration and, for all Borel functions $F : \R \to \C$, the operator
\[
F(\opL) = \int_\R F(\lambda) \, dE_\opL(\lambda)
\]
is bounded on $L^2(\FE)$ if and only if $F$ is an $E_{\opL}$-essentially bounded function.

Characterizing, or just giving nontrivial sufficient conditions for the $L^p$-bounded\-ness of $F(\opL)$ for some $p \neq 2$ in terms of properties of the ``spectral multiplier'' $F$ is a much harder problem. This question has been particularly studied in the case $\opL$ is the Laplace operator on $\R^d$, or some analogue thereof acting on sections of a vector bundle over a smooth $d$-manifold. For the Laplacian $\opL = -\Delta$ on $\R^d$, the classical Mihlin--H\"ormander multiplier theorem \cite{mihlin_multipliers_1956,hrmander_estimates_1960} tells us that $F(\opL)$ is of weak type $(1,1)$ and $L^p$-bounded for all $p \in (1,\infty)$ whenever the multiplier $F$ satisfies the scale-invariant local Sobolev condition
\begin{equation}\label{eq:mhcond}
\|F\|_{L^q_{s,\sloc}} \defeq \sup_{t\geq 0} \|F(t \, \cdot) \,\chi\|_{L^q_s(\R)} < \infty,
\end{equation}
for $q=2$ and some $s > d/2$; here $L^q_s(\R)$ is the $L^q$ Sobolev space of (fractional) order $s$ and $\chi \in C^\infty_c((0,\infty))$ is any nontrivial cutoff (different choices of $\chi$ give rise to equivalent local Sobolev norms). This result is sharp, i.e., the threshold $d/2$ on the order of smoothness $s$ required on the multiplier $F$ cannot be lowered.

Here we are concerned with the case $\opL = \boxb$ is the Kohn Laplacian acting on sections of the bundle $\Lambda^{0,j} \Sphere$ of $(0,j)$-forms $(0 \leq j \leq n-1$) associated to the tangential Cauchy--Riemann complex on the unit sphere $\Sphere$ in $\C^n$, $n \geq 2$.  The sphere $\Sphere$ and the conformally equivalent Heisenberg group have been long studied as models for more general strictly pseudoconvex CR manifolds of hypersurface type \cite{folland_tangential_1972,folland_estimates_1974,geller_laplacian_1980}. The problem of obtaining a spectral multiplier theorem of Mihlin--H\"ormander type for $\boxb$ has been recently considered in \cite{casarino_spectral_sphere}, where the following result is proved.

\begin{thm}[\cite{casarino_spectral_sphere}]\label{thm:ccms}
Let $\boxb$ be the Kohn Laplacian on $(0,j)$-forms on the unit sphere $\Sphere$ in $\C^n$, where $0 < j < n-1$. For all bounded Borel functions $F : \R \to \C$, if $\|F\|_{L^2_{s,\sloc}} < \infty$ for some $s > (2n-1)/2$, then the operator $F(\boxb)$ is of weak type $(1,1)$ and $L^p$-bounded for all $p \in (1,\infty)$, and moreover
\[
\|F(\boxb)\|_{L^1 \to L^{1,\infty}} \leq C \|F\|_{L^2_{s,\sloc}}.
\]
\end{thm}

A key feature of this result is the threshold $(2n-1)/2$ in the smoothness condition, i.e., half the topological dimension $d = 2n-1$ of the sphere $\Sphere$. In fact, 
by means of quite general theorems \cite{cowling_spectral_2001,duong_plancherel-type_2002}, it would be fairly straightforward to prove the above result under the stronger assumption ``$\|F\|_{L^\infty_{s,\sloc}} <\infty$ for some $s > Q/2$'', where $Q=2n$ is the so-called homogeneous dimension associated with the control distance for $\boxb$. The fact that $Q > d$ is connected with the lack of ellipticity of $\boxb$ (cf.\ \cite{folland_neumann_1972,fefferman_subelliptic_1983}) and the problem of obtaining sharp multiplier theorems of Mihlin--H\"ormander type for non-elliptic, subelliptic operators is still widely open (see, e.g., \cite{martini_necessary}). Most of the analysis in \cite{casarino_spectral_sphere}
is devoted to proving a ``weighted Plancherel-type estimate'' that allows weakening the assumption on the multiplier, by replacing $Q/2$ with $d/2$ (and $L^\infty_s$ with $L^2_s$).

The cases $j=0$ and $j=n-1$ are not treated in \cite{casarino_spectral_sphere}. These cases are exceptional because the orthogonal projection onto the kernel of $\boxb$ (which coincides with the Szeg\H{o} projection in the case $j=0$) is not $L^1$-bounded \cite{koranyi_singular_1971}. This constitutes a serious obstruction to the application of the ``standard machinery'' of \cite{cowling_spectral_2001,duong_plancherel-type_2002}, on which \cite{casarino_spectral_sphere} is based, and moreover puts some limits on the results that can be expected. Indeed,  when $0 < j < n-1$, the Bochner--Riesz means $(1-t\boxb)_+^\alpha$ of order $\alpha > (d-1)/2$ are $L^1$-bounded for all $t > 0$ \cite[Theorem 1.2]{casarino_spectral_sphere}. The analogous statement in the case $j\in\{0,n-1\}$ is simply false, independently of the order $\alpha$.

An alternative approach to this problem is developed in \cite{street_heat_2009}, where a multiplier theorem of Mihlin--H\"ormander type for $\boxb$ in the case $j=0$ is proved for a fairly general class of compact CR manifolds. However in \cite{street_heat_2009} the more restrictive smoothness condition ``$\| F\|_{L^2_{s,\sloc}} < \infty$ for some $s > (Q+1)/2$'' is required and the technique used seems not to yield a weak type $(1,1)$ bound.

In contrast, the main result of the present paper, which extends Theorem \ref{thm:ccms} to the missing cases $j = 0$ and $j=n-1$, requires a smoothness condition $s>d/2$ on the multiplier and includes the weak type $(1,1)$ endpoint.

\begin{thm}\label{thm:m_sphere}
Let $\boxb$ be the Kohn Laplacian on $(0,j)$-forms on the unit sphere $\Sphere$ in $\C^n$, where $j \in \{0,n-1\}$. For all bounded Borel functions $F : \R \to \C$, if $\|F\|_{L^2_{s,\sloc}} < \infty$ for some $s > (2n-1)/2$, then the operator $F(\boxb)$ is of weak type $(1,1)$ and $L^p$-bounded for all $p \in (1,\infty)$, and moreover
\[
\|F(\boxb)\|_{L^1 \to L^{1,\infty}} \leq C \|F\|_{L^2_{s,\sloc}}.
\]
\end{thm}

Note that this implies the weak type $(1,1)$ and $L^p$-boundedness for $p \in (1,\infty)$ of the Bochner--Riesz means $(1-t\boxb)_+^\alpha$ for all $\alpha > (d-1)/2$ and $t>0$.

Our proof of Theorem \ref{thm:m_sphere} could be easily adapted to the case of Heisenberg groups, equipped with the standard strictly pseudoconvex structure. In fact, the proof there would be somehow easier because of the translation-invariance and homogeneity of $\boxb$ on Heisenberg groups. However, there is no need to do this, in the sense that the result on the Heisenberg group can be directly obtained from the corresponding result on the sphere by transplantation.

\begin{cor}\label{cor:m_heisenberg}
Theorems \ref{thm:ccms} and \ref{thm:m_sphere} hold also when the sphere $\Sphere$ is replaced by the $(2n-1)$-dimensional Heisenberg group $\Heis_{n-1}$.
\end{cor}

The idea of transplanting estimates from complex spheres to Heisenberg groups has been used several times in the literature (see, e.g., \cite{ricci_transferring_1986,dooley_contraction_1999,dooley_transferring_2002,casarino_transferring_2009}). Here however we propose a different approach, along the lines of \cite{kenig_divergence_1982}, which 
 does not require any group or symmetric space structure on the manifold, or group-invariance of the operator.
This general transplantation technique (Theorem \ref{thm:transplantation}) applies to arbitrary self-adjoint differential operators on a vector bundle over a smooth manifold and allows transplanting weak type as well as strong type bounds.

The same technique, combined with an argument of \cite{martini_necessary}, yields the sharpness of the above multiplier theorems. In fact, thanks to the analyisis of \cite{beals_calculus_1988}, we can prove a more general result for the Kohn Laplacian on any non-Levi-flat CR manifold of hypersurface type (see \cite{beals_calculus_1988,boggess_cr_1991,dragomir_differential_2006} for definitions).

For a general non-negative selfadjoint operator $\opL$ on $L^2(\FE)$ as in \eqref{eq:spres}, define the sharp Mihlin--H\"ormander threshold $\thr(\opL)$ as the infimum of the $s \in (0,\infty)$ such that
\[
\exists C \in (0,\infty) \tc \forall F \in \mathfrak{B} \tc \|F(\opL)\|_{L^{2} \to L^2} + \|F(\opL)\|_{L^{1} \to L^{1,\infty}} \leq C \, \| F \|_{L^2_{s,\sloc}},
\]
where $\mathfrak{B}$ is the set of bounded Borel functions $F : \R \to \C$. Clearly
\[
\thr(\opL) \geq \thr_-(\opL),\]
where $\thr_-(\opL)$ 
is the infimum of the $s \in (0,\infty)$ such that
\[
\forall p \in (1,\infty) \tc \exists C \in (0,\infty) \tc \forall F \in \mathfrak{B} \tc \|F(\opL)\|_{L^p \to L^p} \leq C \, \| F \|_{L^\infty_{s,\sloc}}.
\]
Note that Theorems \ref{thm:ccms} and \ref{thm:m_sphere} and Corollary \ref{cor:m_heisenberg} can be restated as follows:
\[
\thr(\boxb) \leq (2n-1)/2
\]
for the Kohn Laplacian $\boxb$ on a sphere or Heisenberg group of dimension $2n-1$.

\begin{thm}\label{thm:sharp}
Let $M$ be a non-Levi-flat CR manifold of hypersurface type and dimension $2n-1$, 
with a compatible hermitian metric. Let $\boxb$ be any self-adjoint extension of the Kohn Laplacian on $(0,j)$-forms on $M$, where $0 \leq j \leq n-1$. Then
\[
\thr_-(\boxb) \geq (2n-1)/2.
\]
In particular, Theorems \ref{thm:ccms} and \ref{thm:m_sphere} and Corollary \ref{cor:m_heisenberg} are sharp.
\end{thm}

Apart from sharpness and weak type endpoint, another reason of interest for Theorem \ref{thm:m_sphere} is the technique used in its proof. 
To prove that $F(\boxb)$ is of weak type $(1,1)$, here we show that $F(\boxb)$ is a singular integral operator, satisfying the ``averaged H\"ormander condition'' of \cite[Theorem 1]{duong_singular_1999} (see also \cite[Theorem 3.3]{cowling_spectral_2001}):
\begin{equation}\label{eq:gencz}
\sup_{r > 0} \esssup_{y \in \Sphere} \int_{\dist(x,y) \geq r} |K_{F(\boxb) (I-A_r)}(x,y)| \,d\mu(x) < \infty,
\end{equation}
where $\mu$ is the standard measure on $\Sphere$, $\dist$ is the control distance for $\boxb$, $K_T$ denotes the integral kernel of an operator $T$ and $\{A_r\}_{r>0}$ is some ``approximate identity'' (as $r \downarrow 0$) satisfying, among other things, the uniform bound
\begin{equation}\label{eq:uniforml1bound}
\sup_{r>0} \|A_r\|_{1 \to 1} < \infty.
\end{equation}
In other works on spectral multipliers, this approximate identity is constructed as a function of the operator $\opL$ under consideration, such as the ``heat propagator'' $A_r = \exp(-r^2 \opL)$ (see, e.g., \cite{duong_plancherel-type_2002}) or $A_r = \Phi(r \sqrt{\opL})$ for a suitable	Schwartz function $\Phi$ with $\Phi(0)=1$ (cf.\ \cite{cowling_spectral_2001}).
However such choices of $A_r$ are forbidden in the case $\opL = \boxb$ and $j\in \{0,n-1\}$, because the $L^1$-unboundedness of the Szeg\H{o} projection is incompatible with \eqref{eq:uniforml1bound}.

We are then led to looking for an approximate identity $A_r$ outside the functional calculus of $\boxb$, yet somehow related to it, so to be able to prove \eqref{eq:gencz}. Here comes a key observation: in the case $j=0$, the operator $\boxb$ belongs to the joint functional calculus of two commuting differential operators on $\Sphere$, namely, a sublaplacian $L$ and a unit vector field $T$. Differently from $\boxb$, the sublaplacian $L$ does satisfy Gaussian-type heat kernel estimates, so $A_r = \exp(-r^2 L)$ satisfies \eqref{eq:uniforml1bound} (indeed a sharp multiplier theorem for the sublaplacian $L$ has been proved in \cite{cowling_spectral_2011}). In fact, more is true: the $T$-derivatives of the heat kernel of $L$ satisfy Gaussian-type estimates too. Based on these estimates, we can prove a spectral multiplier theorem of Mihlin--H\"ormander type for the joint functional calculus of $L$ and $iT$, which in turn (in combination with the weighted Plancherel-type estimates from \cite{casarino_spectral_sphere}) allows us to derive \eqref{eq:gencz} with $A_r = \exp(-r^2 L)$, 
 whenever $\|F\|_{L^2_{s,\sloc}} < \infty$ for some $s>d/2$.

Spectral multiplier theorems for systems of commuting operators are not new in the literature. Actually, the classical Mihlin--H\"ormander theorem for Fourier multipliers on $\R^d$ can be thought of as a spectral multiplier theorem for the joint functional calculus of the partial derivatives on $\R^d$.
However, in settings other than $\R^d$, most of the available results  (see, e.g., \cite{mller_marcinkiewicz_1995,mller_marcinkiewicz_1996,fraser_marcinkiewicz_1997,veneruso_marcinkiewicz_2000,wrobel_thesis_2014,duong_endpoint_2015,chen_marcinkiewicz_preprint2015}) are of ``Marcinkiewicz type'', i.e., they impose on the multiplier function a condition that is invariant by multiparameter rescaling, and the correspondingly obtained estimates appear not to be suitable to prove a weak type $(1,1)$ bound. Exceptions to this 
are the results of \cite{sikora_multivariable_2009} and \cite[\S 4]{martini_joint_2012}, where 
a one-parameter rescaling of the multiplier is considered and weak type (1,1) estimates are obtained; however these results are not directly applicable to our system $L,iT$ on the sphere, because \cite{sikora_multivariable_2009} only applies to a product setting,
whereas \cite[\S 4]{martini_joint_2012} applies to left-invariant homogeneous operators on a homogeneous Lie group (this would be enough to deal with the Heisenberg group, but not the sphere).

For this reason, in Section \ref{section:joint} we develop an abstract version of \cite[\S 4]{martini_joint_2012} in the context of doubling metric measure spaces,
 which includes the main result of \cite{sikora_multivariable_2009} as a particular case and could be of independent interest. When applied to a single operator $\opL$, the result of Section \ref{section:joint} essentially reduces to the main result of \cite{hebisch_functional_1995}, where only polynomial decay (of arbitrarily high order) is required on the heat kernel of $\opL$, in place of the usual Gaussian-type exponential decay.
Because of the general character of the argument, we have tried to put minimal assumptions on the system of commuting operators, in order to obtain a statement that encompasses many different situations; we refer to Section \ref{section:joint} for an extensive discussion and examples.
 The resulting multivariate spectral multiplier theorem of Mihlin--H\"ormander type (Theorem \ref{thm:multipliers}) is sufficiently strong to serve as a base for our sharp Theorem \ref{thm:m_sphere} and we expect that other similar applications can be found in the future.

\smallskip

Some general remarks about notation are in order. The letter $C$ and variants such as $C_p$ will denote a finite positive quantity that may change from place to place. For any two quantities $A,B$, we also write $A \lesssim B$ instead of $A \leq C B$; moreover $A \sim B$ is the same as ``$A \lesssim B$ and $B \lesssim A$''. $\chr_{U}$ denotes the characteristic function of a set $U$.

\subsection*{Acknowledgments}
The author wishes to thank Valentina Casarino, Michael G.\ Cowling, and Adam Sikora for several discussions on the subject of this work, and Jonathan Bennett for suggestions about the presentation.

The author gratefully acknowledges the support of the Deutsche Forschungsgemeinschaft (project MA 5222/2-1) during his stay at the Christian-Albrechts-Universit\"at zu Kiel (Germany), where this work was initiated. The support of GNAMPA (Progetto 2014 ``Moltiplicatori e proiettori spettrali associati a Laplaciani su sfere e gruppi nilpotenti'' and Progetto 2015 ``Alcune questioni di analisi armonica per operatori differenziali non ellittici'') is also gratefully acknowledged.

\section{Unitary group action and joint spectral decomposition}

This and the next two sections are devoted to the proof of Theorem \ref{thm:m_sphere}. Indeed we need only to discuss the case $j=0$, i.e., the case of the Kohn Laplacian $\boxb$ acting on (scalar-valued) functions on the sphere. In fact, by means of Geller's Hodge star operator (\cite[p.\ 5]{geller_laplacian_1980}; see also \cite[Remark 4.6]{casarino_spectral_sphere}), it is easily seen that the case $j=n-1$ in Theorem \ref{thm:m_sphere} can be reduced to the case $j=0$. 

The Kohn Laplacian $\boxb$ on $\Sphere$ is invariant by the action of the unitary group $\group{U}(n)$. It is therefore natural to exploit the representation theory of $\group{U}(n)$ for the analysis of $\boxb$, as it has been done in great detail in \cite{folland_tangential_1972}. Here we just recall the main results that will be of use later.

As it is well-known (see, e.g., \cite{folland_tangential_1972,cowling_spectral_2011}), the decomposition into irreducible representations of the natural representation of $\group{U}(n)$ on $L^2(\Sphere)$ is multiplicity-free and is given by
\[
L^2(\Sphere) = \bigoplus_{p,q \geq 0} \mathcal{H}_{pq},
\]
where $\mathcal{H}_{pq}$ is the space of $(p,q)$-bihomogeneous complex spherical harmonics (denoted by $\Phi_{pq0}$ in \cite{folland_tangential_1972}). By Schur's Lemma, all $\group{U}(n)$-equivariant operators $R$ on $L^2(\Sphere)$ preserve this decomposition and are scalar when restricted to each $\mathcal{H}_{pq}$̧, i.e., $R|_{\mathcal{H}_{pq}} = \lambda_{pq}^R \id_{\mathcal{H}_{pq}}$ for some $\lambda_{pq}^R \in \C$. In particular, all such operators commute.

Let the sublaplacian $L$ and the unit vector field $T$ on $\Sphere$ be defined as in \cite{cowling_spectral_2011} and set $U=-2(n-1)iT$. The operators $\boxb,L,U$ are $\group{U}(n)$-equivariant, therefore they have a joint spectral decomposition in terms of complex spherical harmonics.

\begin{prp}\label{prp:spectralinfo}
For all $p,q \in \N$,
\begin{align}
\label{eq:dim} \dim \mathcal{H}_{pq} &= \frac{p+q+n-1}{n-1} \binom{p+n-2}{n-2} \binom{q+n-2}{n-2} \\
\label{eq:avL} \lambda_{pq}^L &= 4pq + 2(n -1)(p + q), \\
\label{eq:avU} \lambda_{pq}^U &= 2(n-1)(p-q),\\
\label{eq:avB} \lambda_{pq}^{\boxb} &= 2q(p + n-1).
\end{align}
In particular
\begin{equation}\label{eq:oprelation}
2\boxb = L-U,
\end{equation}
and moreover, for all $p,q\in \N$,
\begin{equation}\label{eq:avestimate}
\lambda_{pq}^{\boxb} \neq 0 \quad\Rightarrow\quad \lambda_{pq}^{\boxb} \leq \lambda_{pq}^L \leq (n+1) \lambda_{pq}^{\boxb}.
\end{equation}
\end{prp}
\begin{proof}
\eqref{eq:dim} is in \cite[Corollary 2.6]{cowling_spectral_2011}, \eqref{eq:avL} and \eqref{eq:avU} are in \cite[Section 4]{cowling_spectral_2011}, while \eqref{eq:avB} is in \cite[Theorem 6]{folland_tangential_1972}. From these expressions we immediately obtain \eqref{eq:avestimate} and
\[
2\lambda_{pq}^{\boxb} = \lambda_{pq}^L - \lambda_{pq}^U,
\]
whence \eqref{eq:oprelation}.
\end{proof}

\section{Heat kernel estimates and a non-sharp multiplier theorem}

Let $\mu$ be the standard hypersurface measure on $\Sphere \subseteq \C^n$ and $\dist$ denote the control distance for $\boxb$. We refer to \cite[Section 3]{casarino_spectral_sphere} for precise definitions and discussion of their main properties. Here we just recall that $\dist$ is $\group{U}(n)$-invariant and
\begin{equation}
\dist(z,w) \sim |1-\langle z,w\rangle|^{1/2},
\end{equation}
where $\langle \cdot , \cdot \rangle$ denotes the standard Hermitian inner product on $\C^n$.
Moreover, if $V(r)$ denotes the $\mu$-measure of any $\dist$-ball of radius $r \in [0,\infty)$, then
\begin{equation}\label{eq:srvol}
V(r) \sim \min\{1,r^Q\},
\end{equation}
where $Q = 2n$. In the language of Section \ref{section:joint}, this tells us that $(\Sphere,\dist,\mu)$ is a doubling metric measure space of homogeneous dimension $Q$ and displacement parameter $0$.

Let us introduce some weighted mixed Lebesgue norms for Borel functions $K : \Sphere \times \Sphere \to \C$, that will be repeatedly used: for all $p \in [1,\infty]$ and $\beta \in \R$,
\[
\vvvert K\vvvert_{p,\beta,r} = \esssup_y  V(r)^{1/p'} \|K(\cdot,y) \, (1+\dist(\cdot,y)/r)^\beta \|_{L^p(\Sphere)}.
\]
Here $p' = p/(p-1)$ denotes the conjugate exponent to $p$.

As already mentioned in the introduction, a key ingredient of our proof is the fact that the heat kernel of $L$, together with its $U$-derivatives, satisfies Gaussian-type estimates.

\begin{prp}\label{prp:gaussianestimates}
There exists $b>0$ such that, for all $k \in \N$, there exists $C_k > 0$ such that, for all $t > 0$ and $x,y \in \Sphere$,
\[
|K_{U^k e^{-tL}}(x,y)| \leq C_k \, t^{-k} \, V(t^{1/2})^{-1} \exp(-b \,\dist(x,y)^2/t).
\]
\end{prp}
\begin{proof}
These estimates are well known, at least for for $t \leq 1$ (see, e.g., \cite[Theorem 3]{jerison_estimates_1986} or \cite[\S 4]{jerison_subelliptic_1987}). On the other hand, since $\Sphere$ is compact and $\dist$ is bounded, the estimates for $t \geq 1$ follow from the uniform bound
\[
\vvvert K_{U^k e^{-tL}} \vvvert_{\infty,0,t^{1/2}} \leq C_k \, t^{-k},
\]
which in turn is easily proved by $L^2$-spectral theory (cf.\ \cite[p.\ 630 top]{cowling_spectral_2011} or \cite[p.\ 24 bottom]{casarino_spectral_sphere}):
\[\begin{split}
\vvvert K_{U^k e^{-tL}} \vvvert_{\infty,0,t^{1/2}} &= V(t^{1/2}) \, \| U^k e^{-tL} \|_{1 \to \infty} \\
&\leq V(t^{1/2}) \, \| U^{k} e^{-tL/2} \|_{2 \to \infty} \, \| e^{-tL/2} \|_{1 \to 2}
\end{split}\]
and, for all $h \in \N$,
\[\begin{split}
\| U^{h} e^{-tL/2} \|_{1 \to 2}^2 &= \| U^{h} e^{-tL/2} \|_{2 \to \infty}^2\\
&= \sum_{p,q \in \N} (\lambda_{pq}^U)^{2h} \, e^{-t \lambda_{pq}^L} \dim\mathcal{H}_{pq} / \mu(\Sphere) \\
&\leq t^{-2h} \sum_{p,q \in \N} (t\lambda_{pq}^L)^{2h} \, e^{-t \lambda_{pq}^L} \dim\mathcal{H}_{pq} / \mu(\Sphere) \\
&\lesssim t^{-2h} \sum_{\ell \in \N} (t \ell^2)^{2h} e^{-t\ell^2} \ell^{Q-1} \\
&\lesssim \, t^{-2h} \, V(t^{1/2})^{-1}
\end{split}\]
by Proposition \ref{prp:spectralinfo}, eq.\ \eqref{eq:srvol} and \cite[eq.\ (21)]{cowling_spectral_2011}.
\end{proof}

Thanks to the above bounds on the heat kernel of $L$ and its $U$-derivative, we can apply Theorem \ref{thm:multipliers} to the system $(L,U)$ and isotropic dilations $\epsilon_r(\lambda) = r^2 \lambda$ on $\R^2$ (cf. Remark \ref{rem:ex}\ref{en:ex2}).

\begin{cor}\label{cor:l2_system}
For all $\beta \geq 0$, $\epsilon,R > 0$, and $G : \R^2 \to \C$ with $\supp G \subseteq [-R^2,R^2]^2$,
\[
\vvvert K_{G(L,U)}\vvvert_{2,\beta,R^{-1}} \leq C_{\beta,\epsilon} \, \|G(R^2\,\cdot) \|_{L^\infty_{\beta+\epsilon}(\R^2)}.
\]
\end{cor}
\begin{proof}
This is an instance of Theorem \ref{thm:multipliers}\ref{en:multipliers_l2} applied to the system $(L,U)$.
\end{proof}

From these estimates we can now easily derive a (non-sharp) spectral multiplier theorem for $\boxb$. Define, for all $t>0$, the operator $A_t$ by
\[
A_t = \exp(-t^2 L).
\]

\begin{cor}\label{cor:nonsharp}
The following holds.
\begin{enumerate}[label=(\roman*)]
\item\label{en:nonsharp_l2} For all $\beta \geq 0$, $\epsilon,R,t > 0$, and $F : \R \to \C$ with $\supp F \subseteq [R/16,R]$,
\[
\vvvert K_{F(\sqrt\boxb) \,(1-A_t)} \vvvert_{2,\beta,R^{-1}} \leq C_{\beta,\epsilon} \,
\|F(R\, \cdot) \|_{L_{\beta+\epsilon}^\infty(\R)} \, \min\{1,(Rt)^2\}.
\]
\item\label{en:nonsharp_l1} For all $\beta \geq 0$, $\epsilon,R,r,t > 0$, and $F : \R \to \C$ with $\supp F \subseteq [R/16,R]$,
\begin{multline*}
\esssup_y \int_{\dist(x,y) \geq r} |K_{F(\sqrt\boxb) \,(1-A_t)}(x,y)| \,dx \\
\leq C_{\beta,\epsilon}\,  \|F(R \,\cdot) \|_{L_{\beta+Q/2+\epsilon}^\infty(\R)} \, \min\{1,(Rt)^2\} \, (1+Rr)^{-\beta}.
\end{multline*}
\item\label{en:nonsharp_wl1} For all $\epsilon > 0$ and $F : \R \to \C$,
\[
\|F(\sqrt{\boxb})\|_{L^1 \to L^{1,\infty}} \leq C_{\epsilon} \, \| F\|_{L^\infty_{Q/2+\epsilon,\sloc}} .
\]
\end{enumerate}
\end{cor}
\begin{proof}
\ref{en:nonsharp_l2}. Let
\[
G(\lambda_1,\lambda_2) = F(\sqrt{(\lambda_1-\lambda_2)_+/2}) \, (1-\exp(-t^2 \lambda_1)) \, \eta((\lambda_1-\lambda_2)/(2\lambda_1)),
\]
where $\eta \in C^\infty_c((0,2))$ and $\eta = 1$ on $[1/(n+1),1]$. By \eqref{eq:oprelation} and \eqref{eq:avestimate}, it is easily seen that
\[
G(L,U) = F(\sqrt{\boxb}) \, (1-A_t)
\]
and moreover, for all $s \in [0,\infty)$,
\[
\|G(R^2 \,\cdot) \|_{L^\infty_s(\R^2)} \leq C_s \, \|F(R \,\cdot) \|_{L^\infty_s(\R)} \, \min\{ 1, (Rt)^2\},
\]
hence \ref{en:nonsharp_l2} follows by Corollary \ref{cor:l2_system}.

\ref{en:nonsharp_l1}. This follows from \ref{en:nonsharp_l2}, since
\[
\esssup_y \int_{\dist(x,y) \geq r} |K_{F(\sqrt\boxb) \,(1-A_t)}(x,y)| \,dx \leq (1+Rr)^{-\beta} \vvvert K_{F(\sqrt\boxb) \,(1-A_t)} \vvvert_{1,\beta,R^{-1}}
\]
and
\[
\vvvert K_{F(\sqrt\boxb) \,(1-A_t)} \vvvert_{1,\beta,R^{-1}} \leq C_{\beta,\epsilon} \, \vvvert K_{F(\sqrt\boxb) \,(1-A_t)} \vvvert_{2,\beta+Q/2+\epsilon/2,R^{-1}}
\]
by H\"older's inequality \eqref{eq:hoelderweight}.

 by H\"older's inequality and \cite[Lemma 4.4]{duong_plancherel-type_2002}.

\ref{en:nonsharp_wl1}. Suppose first that $F(0)=0$. Note that, by Proposition \ref{prp:gaussianestimates}, the operators $A_t$ satisfy the ``Poisson-type bounds'' of \cite[eq.\ (2) and (3)]{duong_singular_1999}. A simple adaptation of the argument in \cite[Proof of Theorem 3.1]{duong_plancherel-type_2002} (see also the proof of Theorem \ref{thm:multipliers}\ref{en:multipliers_mh} below), exploiting \ref{en:nonsharp_l1} in place of \cite[eq.\ (4.18)]{duong_plancherel-type_2002}, shows that $F(\sqrt{\boxb})$ is of weak type $(1,1)$, with
\[
\|F(\sqrt{\boxb})\|_{L^1 \to L^{1,\infty}} \leq C_{\epsilon} \, \| F \|_{L^\infty_{Q/2+\epsilon,\sloc}}.
\]
In particular, if we take $F = \chr_{(0,\infty)}$, then we recover the weak type (1,1) of the Szeg\H{o} projection $\chr_{\{0\}}(\boxb) = I - \chr_{(0,\infty)}(\boxb)$ on the unit sphere \cite{koranyi_singular_1971}.

For a general $F$, it is then sufficient to split
\[
F(\sqrt{\boxb}) = F(0) \, \chr_{\{0\}}(\boxb) + (F \chr_{\R \setminus \{0\}})(\sqrt{\boxb})
\]
and apply what we have just proved to the two summands.
\end{proof}

\section{Sharpening the result}

Here we show how the weighted Plancherel estimates proved in \cite{casarino_spectral_sphere} can be used to sharpen the multiplier theorem given by Corollary \ref{cor:nonsharp}\ref{en:nonsharp_wl1} and obtain Theorem \ref{thm:m_sphere}.
As in \cite[\S 5]{casarino_spectral_sphere}, define the weight $\weight : \Sphere \times \Sphere \to [0,\infty)$ by
\[
\weight(w,z) = |1-|\langle w,z\rangle|^2|^{1/2}.
\]
Here are some basic properties of $\weight$ that will be of use.

\begin{lem}
For all $r>0$ and $\alpha,\beta \geq 0$ such that $\alpha+\beta > Q$ and $\alpha < 2n-2$, and for all $w \in \Sphere$,
\begin{equation}\label{eq:weightsint}
\int_{\Sphere} (1+\dist(z,w)/r)^{-\beta} (1+\weight(z,w)/r)^{-\alpha} \,d\mu(z) \leq C_{\alpha,\beta} V(r).
\end{equation}
Moreover
\begin{equation}\label{eq:weightsint2}
\int_{\Sphere} (1+\weight(z,w)/r)^{-\alpha} \,d\mu(z) \leq C_{\alpha} \, \min\{1,r^{\alpha}\}.
\end{equation}
for all $r > 0$, $\alpha \in [0,2n-2)$ and $w \in \Sphere$, and
\begin{equation}\label{eq:weightsineq}
\weight(w,z) \leq C \dist(w,z)
\end{equation}
for all $w,z \in \Sphere$.
\end{lem}
\begin{proof}
By $\group{U}(n)$-invariance of $\weight$ and $\dist$, it is not restrictive to assume that $w = (1,0,\dots,0)$. If $z = (z_1,z_2,\dots,z_n)=(z_1,z')$, then
\[
\weight(z,w) = |1-|z_1|^2|^{1/2} = |z'|
\]
and
\[
\dist(z,w) \sim |1-z_1|^{1/2} \gtrsim |z'| + |\Im z_1|^{1/2}
\]
(see \cite[Proposition 3.1]{casarino_spectral_sphere}). This gives \eqref{eq:weightsineq} immediately, and \eqref{eq:weightsint2} follows because
\[
\int_{\Sphere} \weight(z,w)^{-\alpha} \,d\mu(z) < \infty
\]
for $\alpha < 2n-2$.

As for \eqref{eq:weightsint}, the case $r \geq 1$ is trivial because $\Sphere$ is compact, so we may assume $r \leq 1$. Since $\beta > Q-\alpha = (2n-2-\alpha) + 2$ and $2n-2-\alpha > 0$, we can decompose $\beta = \beta_1+\beta_2$ so that $\beta_1 > 2n-2-\alpha$ and $\beta_2 > 2$, hence
\[\begin{split}
\int_{\Sphere} &(1+\dist(z,w)/r)^{-\beta} (1+\weight(z,w)/r)^{-\alpha} \,d\mu(z) \\
&\leq C_{\alpha,\beta} \int_{\Sphere} (1+(|z'|+|\Im z_1|^{1/2})/r)^{-\beta} (1+|z'|/r)^{-\alpha} \,d\mu(z) \\
&\leq C_{\alpha,\beta} \int_{\Sphere} (1+|\Im z_1|^{1/2}/r)^{-\beta_2} (1+|z'|/r)^{-\alpha-\beta_1} \,d\mu(z) \\
&\leq C_{\alpha,\beta} \int_\R (1+|u|^{1/2}/r)^{-\beta_2} \,du  \int_{\R^{2n-2}} (1+|v|/r)^{-\alpha-\beta_1} \,dv\\
&= C_{\alpha,\beta} \,r^{Q}
\end{split}\]
and we are done.
\end{proof}

Let us introduce the following bi-weighted mixed Lebesgue norms for Borel functions $K : \Sphere \times \Sphere \to \C$: for all $p \in [1,\infty]$ and $\alpha,\beta \in \R$,
\[
\vvvert K\vvvert_{p,\beta,\alpha,r} = \esssup_y  V(r)^{1/p'} \|K(\cdot,y) \, (1+\dist(\cdot,y)/r)^\beta \, (1+\weight(\cdot,y)/r)^\alpha \|_{L^p(\Sphere)}.
\]
As in \cite[\S 2]{casarino_spectral_sphere}, for all $N \in \N \setminus \{0\}$ and $F : \R \to \C$ supported in $[0,1]$, let the norm $\|F \|_{N,2}$ be defined by
\[
\|F \|_{N,2} = \left( \frac{1}{N} \sum_{i=1}^N \sup_{\lambda \in \left[\frac{i-1}{N},\frac{i}{N}\right]} |F(\lambda)|^2 \right)^{1/2}.
\]

\begin{prp}\label{prp:weightedplancherel} 
For all $\alpha \in [0,1/2)$, $N \in \N \setminus \{0\}$, $t > 0$, and $F : \R \to \C$ vanishing out of $(0,N)$,
\begin{align}
\label{eq:weightedplancherel_pure} \vvvert K_{F(\sqrt\boxb)} \vvvert_{2,0,\alpha,N^{-1}} &\leq C_{\alpha} \, \|F(N\,\cdot) \|_{N,2} , \\
\label{eq:weightedplancherel_approx} \vvvert K_{F(\sqrt\boxb) \,(1-A_t)} \vvvert_{2,0,\alpha,N^{-1}} &\leq C_{\alpha} \, \|F(N\,\cdot) \|_{N,2} \min\{1,(Nt)^2\}.
\end{align}
\end{prp}
\begin{proof}
We prove only \eqref{eq:weightedplancherel_approx}, the other estimate being similar and easier. By Proposition \ref{prp:spectralinfo} and \cite[Lemma 4.3]{casarino_spectral_sphere}, $K_{F(\sqrt\boxb) \,(1-A_t)}$ is a ``kernel polynomial'' in the sense of \cite[\S 5]{casarino_spectral_sphere}, which satisfies the assumptions of \cite[Proposition 5.3]{casarino_spectral_sphere}. Hence, for all $\theta \in [0,1)$,
\[\begin{split}
\int_{\Sphere} &|K_{F(\sqrt\boxb) \,(1-A_t)}(x,y)|^2 \, \weight(x,y)^\theta \,dx \\
&\leq C_\theta \, N^{Q-1-\theta} \sum_{j=2}^N \max\{ |F(\sqrt{\lambda_{pq}^{\boxb}}) (1-\exp(-t^2 \lambda_{pq}^L))|^2 \tc (j-1)^2 \leq \lambda_{pq}^{\boxb} \leq j^2 \} \\
&\leq C_\theta \, N^{Q-1-\theta} (1-e^{-(n+1) t^2 N^2}) \, \sum_{j=2}^N \max\{ |F(\sqrt{\lambda_{pq}^{\boxb}})|^2 \tc (j-1)^2 \leq \lambda_{pq}^{\boxb} \leq j^2 \}
\end{split}\]
where we have used \eqref{eq:avestimate}. As in \cite[proof of Theorem 1.1, p.\ 25 top]{casarino_spectral_sphere}, we then obtain that, for all $\theta \in [0,1)$,
\[
\int_{\Sphere} |K_{F(\sqrt\boxb) \,(1-A_t)}(x,y)|^2 \, \weight(x,y)^\theta \,dx \leq C_\theta\, N^{Q-\theta} \, \|F(N\,\cdot)\|_{N,2}^2 \, \max\{1,(Nt)^2\}^2.
\]
Since $V(N^{-1}) \sim N^{-Q}$ for $N \in \N \setminus \{0\}$, the conclusion follows by combining the two instances of the above estimate corresponding to $\theta =0$ and $\theta = 2\alpha$.
\end{proof}

Fix $\xi \in C^\infty_c((-1/16,1/16))$ nonnegative with
\begin{equation}\label{eq:moments}
\int_\R \xi(t) \,dt = 1 \qquad\text{and}\qquad \int_\R t^k \,\xi(t) \,dt = 0 \text{ for $k=1,\dots,2Q$}.
\end{equation}

As in \cite{cowling_spectral_2001,duong_plancherel-type_2002}, we first deal with the operator $(F * \xi)(\boxb)$ corresponding to a smoothened version of the multiplier $F$.

\begin{prp}\label{prp:improvedweightedconv}
The following holds.
\begin{enumerate}[label=(\roman*)]
\item\label{en:improvedweightedconv_l2} For all $\beta \geq 0$, $\alpha \in [0,1/2)$, $N \in \N \setminus \{0\}$, $\epsilon,t > 0$, and $F : \R \to \C$ with $\supp F \subseteq [N/4,3N/4]$,
\[
\vvvert K_{(F*\xi)(\sqrt\boxb) \,(1-A_t)} \vvvert_{2,\beta,\alpha,N^{-1}} \leq C_{\beta,\epsilon} \, \|F(N\,\cdot) \|_{L_{\beta+\epsilon}^2(\R)} \, \min\{1,(Nt)^2\}.
\]
\item\label{en:improvedweightedconv_l1} For all $\beta \geq 0$, $N \in \N \setminus \{0\}$, $\epsilon,r,t > 0$, and $F : \R \to \C$ with $\supp F \subseteq [N/4,3N/4]$,
\begin{multline*}
\esssup_y \int_{\dist(x,y) \geq r} |K_{(F*\xi)(\sqrt\boxb) \,(1-A_t)}(x,y)| \,dx \\
\leq C_{\beta,\epsilon}\,  \|F(N\,\cdot) \|_{L_{\beta+d/2+\epsilon}^2(\R)} \min\{1,(Nt)^2\} (1+Nr)^{-\beta}.
\end{multline*}
\item\label{en:improvedweightedconv_wl1} For all $\epsilon > 0$ and $F : \R \to \C$ with $\supp F \subseteq [1/2,\infty)$,
\[
\|(F*\xi)(\sqrt{\boxb})\|_{L^1 \to L^{1,\infty}} \leq C_{\epsilon} \, \| F \|_{L^2_{d/2+\epsilon,\sloc}}.
\]
\end{enumerate}
\end{prp}
\begin{proof}
\ref{en:improvedweightedconv_l2}. By Corollary \ref{cor:nonsharp}\ref{en:nonsharp_l2} and Young's inequality, we obtain that
\[
\vvvert K_{(F*\xi)(\sqrt\boxb) \,(1-A_t)} \vvvert_{2,\beta,N^{-1}} \leq C_{\beta,\epsilon} \|F(N\,\cdot) \|_{L_{\beta+\epsilon}^\infty(\R)} \, \min\{1,(Nt)^2\}.
\]
for all $\beta \geq 0$, $N \in \N \setminus \{0\}$, $\epsilon,t > 0$, and $F : \R \to \C$ with $\supp F \subseteq [N/8,7N/8]$; by Sobolev embedding and the inequality \eqref{eq:weightsineq} we have in particular that
\[
\vvvert K_{(F*\xi)(\sqrt\boxb) \,(1-A_t)} \vvvert_{2,\beta,\alpha,N^{-1}} \leq C_{\beta,\alpha,\epsilon} \|F(N\,\cdot) \|_{L_{\beta+\alpha+1/2+\epsilon}^2(\R)} \, \min\{1,(Nt)^2\}.
\]
for all $\alpha,\beta \geq 0$, $N \in \N \setminus \{0\}$, $\epsilon,t > 0$, and $F : \R \to \C$ with $\supp F \subseteq [N/8,7N/8]$.

On the other hand, by Proposition \ref{prp:weightedplancherel} and \cite[eq.\ (4.9)]{duong_plancherel-type_2002} we have that
\[
\vvvert K_{(F*\xi)(\sqrt\boxb) \,(1-A_t)} \vvvert_{2,0,\alpha,N^{-1}} \leq C_\alpha \|F(N \,\cdot) \|_{2} \, \min\{1,(Nt)^2\}
\]
for all $\alpha \in [0,1/2)$, $N \in \N \setminus \{0\}$, $t > 0$, and $F : \R \to \C$ with $\supp F \subseteq [N/8,7N/8]$. Interpolation of the last two inequalities (cf., e.g., \cite[proof of Lemma 1.2]{mauceri_vectorvalued_1990} or \cite[proof of Proposition 13]{martini_grushin_2012}) gives the conclusion.

\ref{en:improvedweightedconv_l1}. Note that
\begin{multline*}
\esssup_y \int_{\dist(x,y) \geq r} |K_{(F*\xi)(\sqrt\boxb) \,(1-A_t)}(x,y)| \,dx \\
\leq \vvvert K_{(F*\xi)(\sqrt\boxb) \,(1-A_t)} \vvvert_{1,\beta,0,N^{-1}} (1+Nr)^{-\beta}.
\end{multline*}
Moreover, for all nonnegative $\alpha \in (1/2-\epsilon/2,1/2)$, since $d/2+\epsilon/2+\alpha > d/2+1/2=Q/2$, H\"older's inequality and \eqref{eq:weightsint} give that
\[
\vvvert K_{(F*\xi)(\sqrt\boxb) \,(1-A_t)} \vvvert_{1,\beta,0,N^{-1}} \leq C_{\alpha,\beta,\epsilon} \, \vvvert  K_{(F*\xi)(\sqrt\boxb) \,(1-A_t)} \vvvert_{2,\beta+d/2+\epsilon/2,\alpha,N^{-1}}.
\]
The conclusion follows by \ref{en:improvedweightedconv_l2}.

\ref{en:improvedweightedconv_wl1}. This follows from \ref{en:improvedweightedconv_l1} in the same way as Corollary \ref{cor:nonsharp}\ref{en:nonsharp_wl1} follows from Corollary \ref{cor:nonsharp}\ref{en:nonsharp_l1}.
\end{proof}

What is missing is $(F-F*\xi)(\sqrt{\boxb})$; however this part satisfies even better estimates.

\begin{prp}\label{prp:improvedmissingconv}
For all $\epsilon > 0$ and $F : \R \to \C$ with $\supp F \subseteq [1/2,\infty)$,
\[
\|(F-F*\xi)(\sqrt{\boxb})\|_{L^1 \to L^1} \leq C_{\epsilon} \, \| F \|_{L^2_{d/2+\epsilon,\sloc}}.
\]
\end{prp}
\begin{proof}
We follow the argument in \cite[proof of Theorem 3.6]{cowling_spectral_2001} (cf.\ also \cite[proof of Theorem 3.2]{duong_plancherel-type_2002}).

Set $\epsilon' = \min\{1/4,\epsilon\}$ and choose $\alpha \in [0,1/2)$ such that $d/2+\epsilon'+\alpha > Q/2$. Choose $\eta \in C^\infty_c((0,\infty))$ such that $\supp \eta \subseteq [1/4,3/4]$ and $\sum_{k \in \Z} \eta_k = 1$ on $(0,\infty)$, where $\eta_k = \eta(2^{-k} \cdot)$. In particular, if $F_k = \eta_k F$, then $F = \sum_{k \in \N} F_k$.
Set $\xi_k = 2^k \xi(2^k \cdot)$.
By \eqref{eq:l1norm}, H\"older's inequality, \eqref{eq:weightsint2}, \eqref{eq:srvol}, the ``Plancherel estimate'' \eqref{eq:weightedplancherel_pure}, and \cite[Proposition 4.6]{duong_plancherel-type_2002} (which applies because of \eqref{eq:moments}), for all $k \in \N$,
\[\begin{split}
\|(F_k-F_k*\xi)(\sqrt{\boxb})\|_{L^1 \to L^1} &= \vvvert (F_k-F_k*\xi)(\sqrt{\boxb}) \vvvert_{1,0,0,2^{-k}} \\ 
&\leq C_\epsilon \, 2^{-k\alpha} \, V(2^{-k})^{-1/2} \, \vvvert (F_k-F_k*\xi)(\sqrt{\boxb}) \vvvert_{2,0,\alpha,2^{-k}} \\
&\leq C_\epsilon \, 2^{k(Q/2-\alpha)} \, \| F_k(2^k \cdot) - F_k(2^k \cdot) * \xi_k \|_{2^k,2} \\
&\leq C_\epsilon \, 2^{k(Q/2-\alpha-d/2-\epsilon')} \, \| F_k(2^k \cdot) \|_{L^2_{d/2+\epsilon'}} \\
&= C_\epsilon \, 2^{k(Q/2-\alpha-d/2-\epsilon')} \, \| F(2^k \cdot) \, \eta \|_{L^2_{d/2+\epsilon}}.
\end{split}\]
The conclusion follows by summing over $k \in \N$.
\end{proof}

\begin{proof}[Proof of Theorem \ref{thm:m_sphere}]
By combining Propositions \ref{prp:improvedweightedconv}\ref{en:improvedweightedconv_wl1} and \ref{prp:improvedmissingconv}, we have that
\[
\|F(\sqrt{\boxb})\|_{L^1 \to L^{1,\infty}} \lesssim  \| F \|_{L^2_{d/2+\epsilon,\sloc}}
\]
for all $\epsilon > 0$ and $F : \R \to \C$ with $\supp F \subseteq [1/2,\infty)$. On the other hand, since the eigenvalues $\lambda_{pq}^{\boxb}$ of $\boxb$ are nonnegative integers by \eqref{eq:avB} and $\chr_{\{0\}}(\boxb)$ is of weak type $(1,1)$ (see Corollary \ref{cor:nonsharp}\ref{en:nonsharp_wl1}),
\[
\|F(\sqrt{\boxb})\|_{L^1 \to L^{1,\infty}} = \| F(0) \, \chr_{\{0\}}(\boxb)\|_{L^1 \to L^{1,\infty}} \lesssim |F(0)|
\]
for all $F : \R \to \C$ with $\supp F \subseteq (-\infty,1)$. Hence a partition-of-unity argument yields
\[
\|F(\sqrt{\boxb})\|_{L^1 \to L^{1,\infty}} \leq C_{\epsilon} \,  \| F \|_{L^2_{d/2+\epsilon,\sloc}}.
\]
Interpolation of this bound with the trivial estimate
\[
\|F(\sqrt{\boxb})\|_{L^2 \to L^{2}} \leq \|F\|_{\infty} \leq C_{\epsilon} \,  \| F \|_{L^2_{d/2+\epsilon,\sloc}}
\]
gives $L^p$-boundedness of $F(\sqrt{\boxb})$ for $p \in (1,2)$; the case $p \in (2,\infty)$ follows by applying this result to $\overline{F}$ in place of $F$.
\end{proof}

\section{Transplantation to the Heisenberg group and sharpness}

In this section we prove Corollary \ref{cor:m_heisenberg} and Theorem \ref{thm:sharp}
via transplantation. 
Our approach is an extension of the technique of \cite{kenig_divergence_1982}, which is based on perturbation theory for self-adjoint operators. In \cite{kenig_divergence_1982} scale-invariant $L^p$-bounds for the functional calculus of a self-adjoint differential operator $D$ on a $d$-manifold $M$ are transplanted to analogous bounds for the homogeneous constant-coefficient differential operator $D_0$ on $\R^d$ corresponding to the principal symbol of $D$ (with respect to a choice of coordinates) at an arbitrary point of $M$.

The results of \cite{kenig_divergence_1982} do not apply to our situation, because the Kohn Laplacian on the Heisenberg group (that would play the role of $D_0$ above) is not a constant-coefficient operator. 
Indeed the method of ``freezing the coefficients'' is not appropriate for the analysis of operators such as the Kohn Laplacian on a CR manifold (see, e.g., \cite{folland_estimates_1974,stanton_heat_1984,beals_calculus_1988}). For this reason we introduce in Definition \ref{dfn:localmodel} below a generalization of this method, based on a system of (possibly nonisotropic) dilations. Correspondingly, we prove a general transplantation result (Theorem \ref{thm:transplantation}) for an arbitrary self-adjoint differential operator acting on sections of a vector bundle of a smooth manifold.
Finally we apply the general result to the Kohn Laplacian.

Manifolds are assumed to be smooth and second-countable. Vector bundles over manifolds are assumed to be smooth as well. By a hermitian vector bundle we mean a complex vector bundle with a (smooth) hermitian metric. A smooth measure on a manifold is a positive Borel measure whose density with respect to the Lebesgue measure in all coordinate charts is smooth and nowhere vanishing.

If $\FE$ is a vector bundle over $M$ and $U \subseteq M$, we denote by $\FE|_U$ the vector bundle over $U$ obtained by restriction. By $\Trv_d^k$ we denote the trivial bundle over $\R^d$ with fiber $\C^k$, equipped with the standard hermitian metric. Sections of $\Trv_d^k$ will be identified with $\C^k$-valued functions.

By a system of dilations $(\delta_R)_{R>0}$ on $\R^d$ we mean a family of linear automorphisms of $\R^d$ of the form $\delta_R = \exp((\log R) A)$ for some positive self-adjoint linear endomorphism $A$ of $\R^d$; note that $\det \delta_R = R^Q$, where $Q = \tr A$ is the ``homogeneous dimension'' associated with the system of dilations.

\begin{dfn}\label{dfn:localmodel}
Let $M$ be a $d$-manifold, equipped with a smooth measure $\mu$. Let $\FE$ be a hermitian vector bundle of rank $k$ over $M$. Let $D : C^\infty(\FE) \to C^\infty(\FE)$ be a differential operator.
 We say that a differential operator $D_0 : C^\infty(\Trv^k_d) \to C^\infty(\Trv^k_d)$ is a \emph{local model} for $D$ at the point $x \in M$ of order $\gamma \in \R$ with respect to a system of dilations $(\delta_R)_{R>0}$ on $\R^d$ if there exist a coordinate chart $\phi : U \to V \subseteq \R^d$ of $M$ centered at $x$ and an orthonormal frame $X = (X_l)_{l=1}^k$ for $\FE|_U$ such that, if we define
\begin{itemize}
\item $\Psi : C^\infty(\FE|_U) \to C^\infty(\Trv^k_d|_V)$ by
\begin{equation}\label{eq:coordinate_map}
\Psi \left( \sum_{l=1}^k f_l X_l \right) = ((f_l \circ \phi^{-1}) \, a^{1/2})_{l=1}^k,
\end{equation}
where $a \in C^\infty(V)$ is the density with respect to the Lebesgue measure of the push-forward of $\mu$ via $\phi$, and
\item $D_R : C^\infty(\Trv^k_d|_{\delta^{-1}_R(V)}) \to C^\infty(\Trv^k_d|_{\delta^{-1}_R(V)})$, for all $R > 0$, as the differential operator given by
\[
D_R f = R^{\gamma} (\Psi D \Psi^{-1}(f \circ \delta^{-1}_{R})) \circ \delta_R,
\]
\end{itemize}
then,
for all $f \in C^\infty_c(\Trv^k_d)$,
\begin{equation}\label{eq:stronglimit}
\lim_{R \downarrow 0} D_{R} f = D_0 f \qquad\text{in $L^2$}.
\end{equation}
\end{dfn}

Note that the domains $C^\infty(\Trv^k_d|_{\delta^{-1}_R(V)})$ and $C^\infty(\Trv^k_d)$ of the differential operators $D_R$ and $D_0$ may be different. However, if we identify $C^\infty_c(\Trv^k_d|_{\delta^{-1}_R(V)})$ with a subspace of $C^\infty_c(\Trv^k_d)$, then, for all $f \in C^\infty_c(\Trv^k_d)$, we have that $\supp f \subseteq \delta^{-1}_R(V)$ for all sufficiently small $R> 0$, so $D_R f$ is well defined as an element of $C^\infty_c(\Trv^k_d)$ for such small $R$ and the condition \eqref{eq:stronglimit} makes sense.

Similar considerations show that the coordinate chart and the local orthonormal frame in Definition \ref{dfn:localmodel} could be replaced with their restrictions to any smaller open neighborhood of $x$, without changing the limit operator $D_0$ in \eqref{eq:stronglimit}.	Hence two differential operators having the same germ at $x$ would have the same local model.

The definition \eqref{eq:coordinate_map} of the map $\Psi$ involves the density $a$ of the measure on $M$, in such a way that $\Psi|_{C^\infty_c(\FE|_U)}$ extends to an isometry $L^2(\FE|_U) \to L^2(\Trv^k_d|_V)$. This corresponds to the fact that our functional calculi, based on the spectral theorem, are initially defined on $L^2$ and equivariant with respect to $L^2$-isometries. Similarly, we require $L^2$-convergence in \eqref{eq:stronglimit}. However, in applications (cf.\ proof of Proposition \ref{prp:boxb_local}), it may happen that the convergence in \eqref{eq:stronglimit} holds in a stronger sense and that the limit operator $D_0$ is independent of the positive smooth function $a$ in \eqref{eq:coordinate_map}.

The system of dilations $(\delta_R)_{R>0}$ is crucial in determining the local model. If we take isotropic dilations $\delta_R(v) = Rv$ and let $\gamma$ be the order of $D$ as a differential operator, then the local model $D_0$ is nothing else than the principal part of the constant-coefficient differential operator obtained by freezing the coefficients of $D$ at $x$ in the chosen coordinates. Hence Theorem \ref{thm:transplantation} extends some results of \cite{kenig_divergence_1982}. On the other hand, our applications involve nonisotropic dilations.

In what follows, we denote by $C_0(\R)$ the space of complex-valued continuous functions on $\R$ vanishing at infinity. Strong convergence of operators is always understood in the sense of the strong $L^2$ operator topology.

\begin{thm}\label{thm:transplantation}
Let $M$ be a $d$-manifold, equipped with a smooth measure. Let $\FE$ be a hermitian vector bundle of rank $k$ over $M$. Let $D : C^\infty(\FE) \to C^\infty(\FE)$ be a formally self-adjoint differential operator. Suppose that $D$ has a local model $D_0 : C^\infty(\Trv^k_d) \to C^\infty(\Trv^k_d)$ at some point $x \in M$ of order $\gamma \in \R$ with respect to some dilations $(\delta_R)_{R>0}$ on $\R^d$. Assume that $D_0$ is essentially self-adjoint on $C^\infty_c(\Trv^k_d)$ and denote its unique self-adjoint extension by $D_0$ as well. Let $\tilde D$ be any self-adjoint extension of $D$. Then, for all $F \in C_0(\R)$ and all neighborhoods $U \subseteq M$ of $x$,
\begin{equation}\label{eq:wl1transplantation}
\| F(D_0) \|_{L^1(\Trv^k_d) \to L^{1,\infty}(\Trv^k_d)} \leq \liminf_{r\downarrow 0} \| P_U F(r^\gamma \tilde D) P_U \|_{L^1(\FE) \to L^{1,\infty}(\FE)}
\end{equation}
and, for all $p \in [1,\infty]$,
\begin{equation}\label{eq:lptransplantation}
\| F(D_0) \|_{L^p(\Trv^k_d) \to L^p(\Trv^k_d)} \leq \liminf_{r\downarrow 0} \| P_U F(r^\gamma \tilde D) P_U \|_{L^p(\FE) \to L^p(\FE)},
\end{equation}
where $P_U$ is the operator of multiplication by $\chr_U$. 
\end{thm}
\begin{proof}
The right-hand sides of \eqref{eq:wl1transplantation} and \eqref{eq:lptransplantation} do not increase if we replace $U$ by a smaller neighborhood of $x$. Hence we may assume that $U$ is open and is the domain of a coordinate chart $\phi : U \to V$ centered at $x$ and of an orthonormal frame $X$ for $\FE|_U$ such that the conditions of Definition \ref{dfn:localmodel} are satisfied. We may also assume that the boundary $\partial U$ has zero measure in $M$ and that $V$ is bounded in $\R^d$. Let the map $\Psi : C^\infty(\FE|_U) \to C^\infty(\Trv^k_d|V)$ and the differential operators $D_R$ ($R > 0$)
be as in Definition \ref{dfn:localmodel}. 

 Take an open set $U_* \supseteq U$ of full measure in $M$ and such that $\phi$ and $X$ can be extended to a coordinate chart $\phi_* : U_* \to V_* \subseteq \R^d$ and an orthonormal frame $X_*$ for $\FE|_{U_*}$. To construct such a $U_*$, take a countable open cover $\{U_n\}_{n\in \N}$ of $M$ such that $U_0=U$, each $U_n$ carries coordinates for $M$ and a local frame for $\FE$, and each $\partial U_n$ has zero measure in $M$, and then define $U_* = \bigcup_{n \in \N}(U_n \setminus \overline{\bigcup_{m<n} U_m})$.

 We can then define $\Psi_* : C^\infty(\FE|_{U_*}) \to C^\infty(\Trv^k_d|_{V_*})$ in the same way as $\Psi$ in \eqref{eq:coordinate_map}, with $\phi_*$ and $X_*$ in place of $\phi$ and $X$. In particular  $\Psi_*|_{C^\infty_c(\FE|_{U_*})}$ extends to an isometry $\tilde \Psi : L^2(\FE) \to L^2(\Trv^k_d|{V_*})$.

Correspondingly we extend the differential operators $D_R$ to (densely defined) self-adjoint operators $\tilde D_R$ on $L^2(\Trv^k_d)$ as follows:
\[
\tilde D_R f = R^\gamma (\tilde D_1 (f \circ \delta^{-1}_R)) \circ \delta_R,
\]
where $\tilde D_1$ is the self-adjoint operator on $L^2(\Trv^k_d) = L^2(\Trv^k_d|_{V_*}) \oplus L^2(\Trv^k_d|_{\R^d \setminus V_*})$ given by 
\[
\tilde D_1 = \left(\begin{array}{c|c}
\tilde\Psi \tilde D \tilde\Psi^{-1} & 0 \\\hline
0 & 0 
\end{array}\right).
\]
In particular $\tilde D_R \beta = D_R \beta$ for all $\beta \in C^\infty_c(\Trv^k_d|_{\delta_R^{-1}(V)})$, and
\[
\lim_{R \downarrow 0} \tilde D_{R} f = D_0 f
\]
in $L^2$ for all $f \in C^\infty_c(\Trv^k_d)$. 

Since $D_0$ is essentially self-adjoint on $C^\infty_c(\Trv^k_d)$, by \cite[\S VIII.1.1, Corollary 1.6]{kato_perturbation_1995} we conclude that $\tilde D_{R}$ converges strongly in the generalized sense to $D_0$ as $n \to \infty$.
Consequently, by \cite[\S VIII.1.1, Corollary 1.4]{kato_perturbation_1995}, the resolvents $(\tilde D_{R_n} - \zeta)^{-1}$ converge strongly to $(D_0 - \zeta)^{-1}$ as $n \to \infty$ for all $\zeta \in \C \setminus \R$. An application of the Stone--Weierstra\ss\ theorem to the class of functions $F \in C_0(\R)$ such that
\begin{equation}\label{eq:strongconvergence}
F(\tilde D_{R_n}) \to F(D_0) \quad\text{and}\quad \overline{F}(\tilde D_{R_n}) \to \overline{F}(D_0) \qquad\text{strongly as $n \to \infty$}
\end{equation}
shows that \eqref{eq:strongconvergence} holds for all $F \in C_0(\R)$.

On the other hand, by construction,
\[
F(\tilde D_1) f = \tilde\Psi F(\tilde D) \tilde\Psi^{-1} f
\]
for all $f \in L^2(\Trv^k_d|_V)$, hence
\[
F(\tilde D_R) f = (\tilde\Psi F(R^{\gamma} \tilde D) \tilde\Psi^{-1} (f \circ \delta^{-1}_{R})) \circ \delta_R
\]
for all $f \in L^2(\Trv^k_d|_{\delta_R^{-1}(V)})$.

We now prove \eqref{eq:wl1transplantation}. Since $F(D_0)$ is $L^2$-bounded, for all $f \in L^1 \cap L^2(\Trv^k_d)$ and $\alpha > 0$, if $f_k \in C^\infty_c(\Trv^k_d)$ and $f_k \to f$ in $L^1 \cap L^2$, then
\[
\mu_0(\{ |F(D_0)f| > \alpha \}) \leq \sup_{r>0} \inf_{\epsilon>0} \liminf_{k \to \infty} \mu_0(B_r \cap \{ |F(D_0) f_k| > \alpha/(1+\epsilon) \})
\]
where $\mu_0$ is the Lebesgue measure and $B_r$ is the closed ball of radius $r$ in $\R^d$ centered at the origin. Hence it is sufficient to prove that, for all compact sets $K \subseteq \Heis$, all  $f \in C^\infty_c(\Trv^k_d)$, and all $\alpha > 0$,
\[
\mu_0(K \cap \{ |F(\boxb^{\Heis})f| > \alpha \}) \leq \kappa_F \| f \|_1 / \alpha 
\]
where $\kappa_F = \liminf_{R \downarrow 0} \kappa_{F,R}$ and $\kappa_{F,R} = \| P_U F(R^{\gamma} \tilde D) P_U \|_{L^1(\FE) \to L^{1,\infty}(\FE)}$. On the other hand, by \eqref{eq:strongconvergence} and Markov's inequality,
\[
\mu_0(K \cap \{ |F(D_0)f| > \alpha \}) \leq \inf_{\epsilon>0} \liminf_{R\downarrow 0} \mu_0(K \cap \{ |F(\tilde D_R)f| > \alpha/(1+\epsilon) \}).
\]
Let $R$ be sufficiently small so that $\tilde K := \supp f \cup K \subseteq \delta_R^{-1}(V)$. Let $\mu$ be the measure on $M$. Let $a \in C^\infty(V)$ be as in Definition \ref{dfn:localmodel}, and set $A_R = \max_{\delta_R(\tilde K)} a^{1/2}$, $B_R = \max_{\delta_R(\tilde K)} a^{-1/2}$. Let $Q$ be the homogeneous dimension associated with the dilations $\delta_R$. Then
\[\begin{split}
\mu_0&(K \cap  \{ |F(\tilde D_R)f| > \alpha \})  \\
&= R^{-Q} \mu_0(\delta_R(K) \cap \{ |\tilde\Psi F(R^{\gamma}\tilde D)\tilde\Psi^{-1}(f\circ \delta^{-1}_R)| > \alpha \}) \\
&\leq B_R^2 R^{-Q} \mu (U \cap \{ |F(R^{\gamma}\tilde D)\Psi^{-1}(f\circ \delta^{-1}_R)| > \alpha/A_R \}) \\
&\leq B_R^2 A_R R^{-Q} \kappa_{F,R} \| \Psi^{-1}(f \circ \delta_R^{-1}) \|_{L^1(\FE)} / \alpha \\
&\leq B_R^2 A^2_R R^{-Q} \kappa_{F,R} \| f \circ \delta_R^{-1} \|_{L^1(\Trv^k_d)} / \alpha \\
&= B_R^2 A^2_R \kappa_{F,R} \| f \|_{L^1(\Trv^k_d)} / \alpha 
\end{split}\]
and in particular, since $\lim_{R \downarrow 0} A_R B_R = 1$,
\[
\liminf_{R \downarrow 0} \mu_0(K \cap \{ |F(\tilde D_R)f| > \alpha \}) \leq \kappa_F \|f\|_{L^1(\Trv^k_d)}/\alpha.
\]

The proof of \eqref{eq:lptransplantation} is similar and follows the lines of the proof of \cite[Theorem 2]{kenig_divergence_1982}, keeping track of the constants.
\end{proof}

\begin{cor}\label{cor:transference_thr}
With the same notation and hypotheses as in Theorem \ref{thm:transplantation}, assume moreover that $\tilde D$ is nonnegative. Then $D_0$ is nonnegative and
\[
\thr(D_0) \leq \thr(\tilde D) \qquad\text{and}\qquad \thr_-(D_0) \leq \thr_-(\tilde D).
\]
\end{cor}
\begin{proof}
We may assume that $\tilde D \neq 0$, otherwise $D_0=0$ by \eqref{eq:stronglimit} and the result is trivial. Let $\Sigma$ be the $L^2$-spectrum of $\tilde D$. Then, from the definition of $\thr$ and the fact that $\|F(\tilde D)\|_{L^2 \to L^2} = \sup_{\lambda \in \Sigma} |F(\lambda)|$ for all continuous $F$, it follows easily that $\thr(\tilde D) \geq 1/2$ (see, e.g., \cite[\S 2.6.2, proof of Theorem 1]{triebel_spaces_1978} or \cite[Proposition 2.3.15]{martini_multipliers_2010}).

Take $s > \thr(\tilde D)$ and let $\kappa_s \in (0,\infty)$ be such that
\begin{equation}\label{eq:thrbound}
\|F(\tilde D) \|_{L^2 \to L^2} + \|F(\tilde D) \|_{L^1 \to L^{1,\infty}} \leq \kappa_s \, \|F\|_{L^2_{s,\sloc}}.
\end{equation}
for all bounded Borel functions $F : \R \to \C$. We show now that a similar inequality holds with $D_0$ in place of $\tilde D$.

Take any $F$ such that $\|F\|_{L^2_{s,\sloc}} < \infty$. Since $s>1/2$, by Sobolev embedding $F$ is continuous and bounded on $(0,\infty)$. 
Arguing as in the proof of Corollary \ref{cor:nonsharp}\ref{en:nonsharp_wl1}, it is not restrictive to assume that $F(0) = 0$. Let $\eta_k \in C_c^\infty((0,\infty))$, for $k \in \Z$, be as in the proof of Proposition \ref{prp:improvedmissingconv} and define $F_N = \sum_{|k|\leq N} F \eta_k$. Theorem \ref{thm:transplantation} and \eqref{eq:thrbound} then give that
\[
\| F_N(D_0) \|_{L^2 \to L^{2}} + \| F_N(D_0) \|_{L^1 \to L^{1,\infty}} \leq \kappa \,  \|F_N  \|_{L^2_{s,\sloc}} \leq C_s \, \kappa_s \, \|F\|_{L^2_{s,\sloc}}.
\]
for all $N \in \N$. On the other hand, $F_N(D_0) \to F(D_0)$ strongly as $N \to \infty$. Consequently
\[
\| F(D_0) \|_{L^2 \to L^{2}} + \| F(D_0) \|_{L^1 \to L^{1,\infty}} \leq C_s \, \kappa_s \, \|F \|_{L^2_{s,\sloc}}.
\]

This proves that $\thr(\tilde D) \geq \thr(D_0)$. The other inequality is proved analogously.
\end{proof}

We now apply these results to the Kohn Laplacian. To this purpose we exploit the analysis of \cite{beals_calculus_1988}.

\begin{prp}\label{prp:boxb_local}
Let $M$ be a CR manifold of hypersurface type and dimension $2n-1$, equipped with a compatible hermitian metric. Let $\boxb$ be the Kohn Laplacian acting on section of the bundle $\Lambda^{0,j} M$ of $(0,j)$-forms, where $j \in \{0,\dots,n-1\}$. Let $\mathcal{J} = \{ J \subseteq \{1,\dots,n-1\} \tc |J|=j\}$ and let $\Trv^\mathcal{J}$ be the trivial bundle on $\R^{2n-1}$ with fiber $\C^\mathcal{J}$. Let $y \in M$ and let $\lambda_1,\dots,\lambda_{n-1}$ be the eigenvalues of the Levi form of $M$ at $p$. Define vector fields on $\R^{2n-1}$ by
\[
U_0 = \frac{\partial}{\partial u_0}, \quad U_k = \frac{\partial}{\partial u_k} + 2 \lambda_k \, u_{n-1+k} \frac{\partial}{\partial u_0},  \quad U_{n-1+k} = \frac{\partial}{\partial u_{n-1+k}} - 2 \lambda_k \, u_{k} \frac{\partial}{\partial u_0}
\]
for $k =1,\dots,n-1$. Let $\boxb^y : C^\infty(\Trv^{\mathcal J}) \to C^\infty(\Trv^{\mathcal J})$ be the differential operator defined by
\begin{equation}\label{eq:boxbmodel}
\boxb^y (f_J)_{J \in \mathcal{J}} = (\boxb^{y,J} f_J)_{J \in \mathcal{J}},
\end{equation}
where
\begin{equation}\label{eq:boxbmodelscalar}
-4 \, \boxb^{y,J} = \sum_{k=1}^{2n-2} U_k^2 +4 i \left(\sum_{k \in J} \lambda_k - \sum_{k \notin J} \lambda_k\right) U_0.
\end{equation}
Then $\boxb^y$ is a local model for $\boxb$ at $p$ of order $2$ with respect to the dilations
\[
\delta_R(u_0,u_1,\dots,u_{2n-2}) = (R^2 u_0,R u_1,\dots,R u_{2n-2})
\]
on $\R^{2n-1}$.
\end{prp}
\begin{proof}
This follows from the results of \cite[Chapter 4]{beals_calculus_1988}. Indeed, if the complex vector fields $Z_k$ and $Z_k^y$ are defined as in \cite[eqs.\ (21.1) and (21.5)]{beals_calculus_1988}, and $a$ is any smooth function on the domain of the $Z_k$, then it is easily seen that, for all $f \in C^\infty_c(\R^{2n-1})$,
\begin{align*}
R (Z_k ( f \circ \delta_R^{-1})) \circ \delta_R &\to Z_k^y f  \\
(a (f \circ \delta_R^{-1})) \circ \delta_R &\to a(0) f 
\end{align*}
as $R \downarrow 0$, in the LF-space topology of $C^\infty_c(\R^{2n-1})$. Consequently, by composition, for all nowhere zero smooth functions $a$ and for all $f \in C^\infty_c(\R^{2n-1})$,
\[
R^2 (a^{-1} Z_k \overline{Z}_l ((af) \circ \delta_R^{-1}) \circ \delta_R \to Z_k^y \overline{Z}_l^y f
\]
as $R \downarrow 0$, in the topology of $C^\infty_c(\R^{2n-1})$, and also in $L^2$. From this it is not difficult to see that the operator $\boxb^y$ given by \cite[eq.\ (21.10)]{beals_calculus_1988} is a local model, in the sense of our Definition \ref{dfn:localmodel}, of the operator $\boxb$ given by \cite[eq.\ (20.43)]{beals_calculus_1988}. In order to conclude, it is sufficient to observe that \cite[eq.\ (22.1)]{beals_calculus_1988}, which corresponds to \eqref{eq:boxbmodel} and \eqref{eq:boxbmodelscalar} above, is obtained from \cite[eq.\ (21.10)]{beals_calculus_1988} by means of the coordinate changes \cite[eqs.\ (21.14) and (21.21)]{beals_calculus_1988}, which commute with the dilations $\delta_R$.
\end{proof}

Note that the vector fields $U_0,U_1,\dots,U_{2n-2}$ of Proposition \ref{prp:boxb_local} are left-invariant with respect to the nilpotent Lie group law on $\R^{2n-1}$ defined by
\begin{multline}\label{eq:grouplaw}
(u_0,u_1,\dots,u_{2n-2}) \cdot (u'_0,u'_1,\dots,u'_{2n-2}) \\
= \Biggl(u_0+u_0'+ 2\sum_{k=1}^{n-1} \lambda_k (u_k' u_{n-1+k} - u_k u_{n-1+k}'),u_1+u_1',\dots,u_{2n-2}+u_{2n-2}'\Biggr)
\end{multline}
(cf.\ \cite[Definition (1.14)]{beals_calculus_1988}). Hence the operators $\boxb^{y,J}$ and $\boxb^y$ are essentially self-adjoint (see, e.g., \cite{nelson_representation_1959} or \cite[Proposition 3.2]{martini_spectral_2011}).

It is now immediate to transplant Theorems \ref{thm:ccms} and \ref{thm:m_sphere} to the Heisenberg group.

\begin{proof}[Proof of Corollary \ref{cor:m_heisenberg}]
The sphere $\Sphere$ is a strictly pseudoconvex CR manifold with a Levi metric \cite[Definition 1.5]{stanton_heat_1984}, so the eigenvalues of the Levi form at each point are all equal to $1$. Hence, if we apply Proposition \ref{prp:boxb_local} to the Kohn Laplacian $\boxb^\Sphere$ on the sphere, we obtain as local model (at any point) an operator that corresponds, in suitable coordinates and trivializations, to the Kohn Laplacian $\boxb^H$ on the Heisenberg group with the standard strictly pseudoconvex structure (see \cite[\S\S 4-5]{folland_estimates_1974}). Consequently $\thr(\boxb^H) \leq \thr(\boxb^\Sphere)$ by Corollary \ref{cor:transference_thr}.
\end{proof}

As for Theorem \ref{thm:sharp}, the following variation of a result of \cite{martini_necessary} will be of use.

\begin{prp}\label{prp:sublap_lowerbound}
Let $G$ be a $2$-step stratified group of topological dimension $d$, $\opL$ a homogeneous sub-Laplacian thereon and $V$ a left-invariant vector field in the second layer of $G$, such that $\opL \geq iV$. Then
\[
\thr_-(\opL-iV) \geq d/2.
\]
\end{prp}
\begin{proof}
The proof follows the lines of the argument in \cite[Section 2]{martini_necessary}, where the case $V	=0$ is treated. Here we just list the main steps and modifications needed.
\begin{itemize}
\item As in \cite[Section 2]{martini_necessary}, let $Q$ be the homogeneous dimension of $G$, $\lie{g}_2 \cong \R^{d_2}$ be the second layer of the Lie algebra of $G$ and $\vecU$ be the corresponding vector of central derivatives. Then $-i V = \beta \cdot \vecU$ for some $\beta \in \lie{g}_2$. Similarly as in \cite[eq.\ (9)]{martini_necessary}, define $\Omega_{\beta,t}^{\chi,\theta}$ as the convolution kernel of $m_t^\alpha(\opL+\beta\cdot\vecU) \, \theta(t \vecU)$.
\item By arguing as in \cite[Proposition 5]{martini_necessary}, one obtains an expression for $\Omega_{\beta,t}^{\chi,\theta}$ analogous to the right-hand side of \cite[eq.\ (10)]{martini_necessary}, where $I^\theta$ is replaced by
\begin{multline*}
I^\theta_\beta(t,s,r,y,v) = \int_{\lie{g}_2^*} \exp(i t \Phi(y,v,\mu)-is\Sigma(y,v)-irR(s,r,y,\mu)) \\ \times \exp(i(1-sr) \langle \mu, \beta \rangle) \, B(sr,\mu) \,\theta(\mu) \,d\mu.
\end{multline*}
\item Hence the argument of \cite[Proposition 7]{martini_necessary} gives that
\[
t^{Q-d/2} \, \Omega_{\beta,t}^{\chi,\theta}(2ty,t^2 v) = e^{i\pi d_1/4} e^{it\Psi_\beta(y,v)} A^{\chi,\theta}(y,v) + O(t^{-1}),
\]
where $\Psi_\beta(y,v) = \Psi(y,v) + \beta \cdot \mu^c(y,v)$.
\item One can then repeat the proof of \cite[Theorem 8]{martini_necessary} to obtain that
\[
\| m_t^\chi(\opL+\beta \cdot \vecU) \|_{p\to p} \geq C_{p,\chi,\beta} \, t^{d(1/p-1/2)}, \qquad p \in [1,2], \, t \geq 1.
\]
\item Comparison of this estimate with \cite[eq.\ (8)]{martini_necessary} gives the conclusion.\qedhere
\end{itemize}
\end{proof}

\begin{proof}[Proof of Theorem \ref{thm:sharp}]
As in Proposition \ref{prp:boxb_local}, let $\lambda_1,\dots,\lambda_{n-1}$ be the eigenvalues of the Levi form at a point $y \in M$. 
Since $M$ is not Levi-flat, we can choose $y \in M$ so that not all $\lambda_j$ are zero. Hence the group law \eqref{eq:grouplaw} defines a $2$-step stratified structure on $\R^{2n-1}$, for which $\opL = -\sum_{k=1}^{2n-2} U_k^2$ is a homogeneous sublaplacian and $U_0$ is a vector field in the second layer. Consequently, by Corollary \ref{cor:transference_thr} and Proposition \ref{prp:sublap_lowerbound}, 
\[
\thr_-(\boxb) \geq \thr_-(\boxb^y) = \max_{J \in \mathcal{J}} \thr_-(\boxb^{y,J}) \geq (2n-1)/2
\]
and we are done.
\end{proof}

\section{An abstract multivariable multiplier theorem}
\label{section:joint}

In this section we prove a multiplier theorem for commuting operators in the setting of a doubling metric-measure space $X$. More precisely, we will consider a system of strongly commuting, possibly unbounded self-adjoint operators $U_1,\dots,U_n$ on $L^2(X)$. Such operators have a joint spectral resolution $E$ on $\R^n$, so
\[
U_j = \int_{\R^n} \lambda_j \,dE(\lambda)
\]
for $j=1,\dots,n$, and a joint functional calculus is defined by
\[
F(U_1,\dots,U_n) = \int_{\R^n} F(\lambda) \,dE(\lambda)
\]
for all Borel functions $F : \R^n \to \C$. In the following we aim at giving a sufficient condition for the weak type $(1,1)$ and $L^p$-boundedness of an operator $F(U_1,\dots,U_n)$ in terms of a smoothness condition of Mihlin--H\"ormander type on the  multiplier $F$:
\begin{equation}\label{eq:mhmulti}
\sup_{r \geq 0} \|(F \circ \epsilon_r) \, \chi \|_{L^\infty_s(\R^n)} < \infty.
\end{equation}
for some $s\geq 0$ sufficiently large. Here $(\epsilon_r)_{r>0}$ is a system of dilations on $\R^n$, $\epsilon_0 : \R^n \to \R^n$ is the function constantly $0$, and $\chi \in C^\infty_c(\R^n \setminus \{0\})$ is a cutoff supported on an annulus.

A model result for us is the classical Mihlin--H\"ormander theorem for Fourier multipliers on $\R^n$, thought of as a multiplier theorem for the joint functional calculus of the partial derivatives $U_j = -i\partial_j$ on $\R^n$. In this case isotropic dilations $\epsilon_r(\lambda) = r\lambda$ are considered and a condition \eqref{eq:mhmulti} of order $s>n/2$ is required.

In the general setting described above, clearly some assumptions on the operators $U_1,\dots,U_n$ are necessary, other than self-adjointness and strong commutativity, in order to prove $L^p$-bounds for $p \neq 2$. This problem has been studied extensively in the case $n=1$, especially for a single nonnegative self-adjoint operator $\opL = U_1$; in this case the choice of dilations $\epsilon_r$ is irrelevant and \eqref{eq:mhmulti} reduces to \eqref{eq:mhcond}. Several assumptions on $\opL$ have been considered in the literature (see, e.g., \cite{hebisch_functional_1995,cowling_spectral_2001,duong_plancherel-type_2002,blunck_hoermander_2003,duong_weighted_2011,kriegler_hoermander_2014,kunstmann_spectral_2015} and references therein), which usually involve estimates for the heat propagator $e^{-t\opL}$ or the wave propagator $\cos(t\sqrt{\opL})$ associated with $\opL$. When bounds on the heat propagator (such as Gaussian-type bounds) are assumed, a classical proof strategy (which appears to originate in the study of group-invariant differential operators on Lie groups; see, e.g., \cite{folland_hardy_1982,hulanicki_functional_1984,de_michele_mulipliers_1987,mauceri_vectorvalued_1990,christ_multipliers_1991,alexopoulos_spectral_1994}) consists in a ``change of variable'', i.e., writing $F(\opL) = G(e^{-t\opL})$ for $G(s) = F(-t^{-1}\log s)$, and then using the particularly favourable bounds on the heat propagator $e^{-t\opL}$ to obtain $L^p$-boundedness of $G(e^{-t\opL})$ whenever $G$ is sufficiently smooth and supported away from $0$; since the ``change of variables'' is given by a smooth function, smoothness properties of $G$ can be reduced to smoothness properties of $F$.

For a system $U_1,\dots,U_n$ of several commuting operators, there seems not to be an obvious standard generalization of heat propagator bounds. In \cite{sikora_multivariable_2009} the case of direct products is considered, where $X = X_1 \times \dots \times X_n$ and each $U_j$ operates on a different factor $X_j$ of the product; there a multiplier theorem for the system $U_1,\dots,U_n$ is proved by assuming bounds for each heat propagator $e^{-tU_j}$ on $X_j$. However there are many systems of commuting operators that are not in ``direct product form''. Numerous examples come from the setting of Lie groups and homogeneous spaces, such as the systems of commuting differential operators associated to Gelfand pairs \cite{helgason_groups_1984,gangolli_harmonic_1988,wolf_harmonic_2007,fischer_nilpotent_2012}. In these cases, one can usually find an operator $\opL = P(U_1,\dots,U_n)$ for some polynomial $P$ such that good bounds hold for the heat propagator $e^{-t\opL}$, as well as for the ``$U_j$-derivatives'' $U_j e^{-t\opL}$ of the propagator. In \cite{martini_spectral_2011,martini_joint_2012} such systems of commuting group-invariant differential operators on Lie groups are studied and, in the case of homogeneous operators on a homogeneous Lie group, a multiplier theorem of Mihlin--H\"ormander type is proved. The basic idea in the proof is again to use a change of variables such as $F(U_1,\dots,U_n) = G(e^{-t\opL},U_1 e^{-t\opL},\dots,U_n e^{-t\opL})$ in order to reduce the functional calculus of the original (unbounded) operators to the functional calculus of operators satisfying good bounds.

What follows can be considered as an extension of the multiplier theorem of \cite[\S 4]{martini_joint_2012} to the setting of abstract operators on a doubling metric-measure space. To this purpose, here we take as an assumption the existence of a ``change of variables'' with suitable smoothness and invertibility properties and such that the operators resulting from this change satisfy suitably good bounds. In order for our result to encompass the various different cases mentioned above, we do not prescribe a specific form for this ``change of variables'', and we state fairly minimal hypotheses that are enough for the argument to work.

So far we have been vague on the ``good bounds'' to be required on the operators resulting from the ``change of variables''. In the case $n=1$, a typical assumption is given by  Gaussian-type heat kernel bounds, i.e., (super)exponential decay in space of the integral kernel of the heat propagator. However in \cite{hebisch_functional_1995} it is shown that polynomial decay (of arbitrarily large order) is sufficient. Following \cite{hebisch_functional_1995}, an analogous polynomial decay assumption is stated in our result below. This assumption is sufficient to prove the weak type $(1,1)$ of $F(U_1,\dots,U_n)$ under a condition of the form \eqref{eq:mhmulti} of order $s>Q/2$, where $Q$ is the ``homogeneous dimension'' of $X$ (see definitions below).

Note that in the smoothness condition \eqref{eq:mhmulti} an $L^\infty$ Sobolev norm is used. It would be interesting to investigate whether sharper results involving other Sobolev norms are possible. In the case $n=1$ (single operator), this problem is extensively discussed in \cite{duong_plancherel-type_2002}, where further assumptions on the operator, such as ``Plancherel-type estimates'', are used to replace $L^\infty_s$ with $L^q_s$ for some $s$. A similar approach is used in \cite{martini_joint_2012} in the case of systems of group-invariant operators on Lie groups. For the sake of simplicity, we will not pursue this here.

\bigskip

Let $(X,\dist,\mu)$ be a doubling metric-measure space of homogeneous dimension $Q$ and displacement parameter $N$. In other words, $(X,\dist)$ is a metric space and $\mu$ is a positive regular Borel measure on $X$, such that all balls $B(x,r) = \{ y \in X \tc \dist(x,y) < r\}$ have finite measure $V(x,r) = \mu(B(x,r))$, and moreover there exist $C',C'' > 0$ such that, for all $x,y \in X$, $\lambda,r \in [0,\infty)$,
\begin{align}
\label{eq:doubling}
V(x,\lambda r) &\leq C' (1+\lambda)^Q \, V(x,r), \\
\label{eq:displacement}
V(y, r) &\leq C'' (1+\dist(x,y)/r)^N \, V(x,r).
\end{align}
The constants $Q,N,C',C''$ contained in the inequalities \eqref{eq:doubling}, \eqref{eq:displacement} will be referred to as ``structure constants'' of the doubling space $(X,\dist,\mu)$.

A basic consequence of the doubling property \eqref{eq:doubling} is the following integrability property: for all $s>Q$ and $y \in X$,
\begin{equation}\label{eq:intstandardweight}
\int_X (1+\dist(x,y)/r)^{-s} \,d\mu(x) \leq C_s \, V(y,r).
\end{equation}

In the following we will consider bounded linear operators $T$ between Lebesgue spaces on $X$ that have an integral kernel, i.e., a locally integrable function $K_T : X \times X \to \C$ such that
\begin{equation}\label{eq:integralkernel}
Tf(x) = \int_X K_T(x,y) \, f(y) \,d\mu(y)
\end{equation}
for all $f \in C_c(X)$ and $\mu$-almost all $x \in X$.
Not all the operators that we are interested in have an integral kernel, so we also consider a weaker notion:
in case the function $K_T$ is just locally integrable on $\{ (x,y) \in X \times X \tc x \neq y\}$ and \eqref{eq:integralkernel} holds for $\mu$-a.e.\ $x \notin \supp f$, then we say that $K_T$ is the off-diagonal kernel of $T$.

In dealing with integral kernels, we will repeatedly use certain mixed weighted Lebesgue norms on functions $K : X \times X \to \C$, defined as follows: for all $p \in [1,\infty]$, $s \in [0,\infty)$, $r \in (0,\infty)$,
\[
\vvvert K\vvvert_{p,s,r} = \esssup_y  V(y,r)^{1/p'} \|K(\cdot,y) \, (1+\dist(\cdot,y)/r)^s \|_{L^p(X)}
\]
where $p'=p/(p-1)$ is the conjugate exponent. We will also write $\vvvert K \vvvert_{p,s}$ in place of $\vvvert K \vvvert_{p,s,1}$. It is worth noting that, if the operator $T$ has integral kernel $K_T$, then
\begin{equation}\label{eq:l1norm}
\|T\|_{1 \to 1} = \vvvert K_T \vvvert_{1,0,r}
\end{equation}
for all $r \in (0,\infty)$.

\begin{thm}\label{thm:multipliers}
Let $U_1,\dots,U_n$ be strongly commuting, possibly unbounded self-adjoint operators on $L^2(X)$, with joint spectrum $\Sigma \subseteq \R^n$. Let $\gamma_1,\dots,\gamma_n \in (0,\infty)$ and define dilations $\epsilon_r$ on $\R^n$ by $\epsilon_r(\lambda_1,\dots,\lambda_n) = (r^{\gamma_1} \lambda_1,\dots,r^{\gamma_n} \lambda_n)$ for all $r \in (0,\infty)$. Let $\Sigma_* = \overline{\bigcup_{r>0} \epsilon_r(\Sigma)}$. Assume that there exists a continuous map $\Psi = (\Psi_1,\dots,\Psi_m) : \R^n \to \R^m$ such that:
\begin{enumerate}[label=\textup{(\Alph*)},ref=\Alph*]
\item\label{en:multipliers_ass_inv} $\Psi(0) \neq 0$ and there exist an open neighborhood $\Omega$ of $\Psi(0)$ in $\R^m$ and a smooth map $\Phi : \Omega \to \R^n$ such that $\Phi \circ \Psi$ is the identity on $\Psi^{-1}(\Omega) \cap \Sigma_*$;
\item\label{en:multipliers_ass_ker} for all $j=1,\dots,m$ and $r>0$, the operator $\Psi_j \circ \epsilon_r(U_1,\dots,U_n)$ has an integral kernel and
\begin{equation}\label{eq:gaussiantype}
\sup_{r>0} \vvvert K_{\Psi_j \circ \epsilon_r(U_1,\dots,U_n)} \vvvert_{2,a,r} <\infty
\end{equation}
for all $a \in [0,\infty)$ and $j=1,\dots,m$.
\end{enumerate}
Then the following holds.
\begin{enumerate}[label=(\roman*)]
\item\label{en:multipliers_plancherel} For all bounded Borel functions $F : \R^{n} \to \C$ supported in $[-1,1]^{n}$, the operator $F \circ \epsilon_r(U_1, \dots, U_n)$ has an integral kernel for all $r \in (0,\infty)$ and
\[
\sup_{r>0} \vvvert K_{F \circ \epsilon_r(U_1, \dots, U_n)} \vvvert_{2,0,r} \leq C \|F\|_{\infty}.
\]
\item\label{en:multipliers_l2} For all $p \in [1,\infty]$, $b,s \in [0,\infty)$ with $s > b+Q(1/p-1/2)_+ + N(1/2-1/p)_+$ and all bounded Borel functions $F : \R^{n} \to \C$ supported in $[-1,1]^{n}$,
\[
\sup_{r>0} \vvvert K_{F \circ \epsilon_r(U_1, \dots, U_n)} \vvvert_{p,b,r} \leq C_{p,s,b} \|F\|_{L^\infty_s}.
\]
\item\label{en:multipliers_schwartz} For all $F : \R^n \to \C$ in the Schwartz class, for all $a\in [0,\infty)$ and $p \in [1,\infty]$,
\[
\sup_{r>0} \vvvert K_{F \circ \epsilon_r(U_1,\dots,U_n)} \vvvert_{p,a,r} < \infty.
\]
\item\label{en:multipliers_l1} For all $s > Q/2$ and all bounded Borel functions $F : \R^{n} \to \C$ supported in $[-1,1]^{n}$,
\[
\sup_{r>0} \| F \circ \epsilon_r(U_1, \dots, U_n) \|_{1 \to 1} \leq C_s \|F\|_{L^\infty_s}.
\]
\item\label{en:multipliers_proj} Let $E$ be the joint spectral resolution of $U_1,\dots,U_n$. Then $E(\{0\})$ is bounded on $L^p(X)$ for all $p \in [1,\infty]$. Moreover $E(\{0\})=0$ if $\mu(X) = \infty$. 
\item\label{en:multipliers_mh} For all $s > Q/2$ and all bounded Borel functions $F : \R^{n} \to \C$, if
\[
\sup_{r>0} \| (F \circ \epsilon_r) \, \chi\|_{L^\infty_s} < \infty
\]
for some cutoff $\chi \in C^\infty_c(\R^{n} \setminus \{0\})$ with $\bigcup_{r>0} \epsilon_r(\{\chi \neq 0\}) = \R^{n} \setminus \{0\}$, then the operator $F(U_1,\dots,U_n)$ is of weak type $(1,1)$ and bounded on $L^p(X)$ for all $p \in (1,\infty)$, and moreover
\[
\|F(U_1,\dots,U_n)\|_{L^1 \to L^{1,\infty}} \leq C_{\chi,s} \sup_{r\geq 0} \, \| (F \circ \epsilon_r) \, \chi\|_{L^\infty_s}.
\]
\end{enumerate}
\end{thm}

\begin{rem}\label{rem:ex}
Here are some examples of applications of the above theorem.
\begin{enumerate}[label=(\roman*)]
\item\label{en:ex0} Let $X = \R^d$ with Euclidean metric and Lebesgue measure, so $Q = d$. Let $\partial_j$ be the $j$th partial derivative, $j=1,\dots,d$. Then Theorem \ref{thm:multipliers} can be applied with $n=d$, $\gamma_j=1$ and $U_j=-i\partial_j$ for $j=1,\dots,d$, and $\Psi : \R^d \to \R^{d+1}$ given by
\[
\Psi(\lambda_1,\dots,\lambda_d) = ( e^{-(\lambda_1^2+\dots+\lambda_d^2)}, \lambda_1 \, e^{-(\lambda_1^2+\dots+\lambda_d^2)}, \dots, \lambda_d \, e^{-(\lambda_1^2+\dots+\lambda_d^2)}).
\]
Note that $F(-i\partial_1,\dots,-i\partial_d)$ is the Fourier multiplier operator on $\R^d$ corresponding to the multiplier $F$. Hence Theorem \ref{thm:multipliers}\ref{en:multipliers_mh} in this case essentially reduces to the classical Mihlin--H\"ormander theorem, with a smoothness condition of order $s>d/2$ on the multiplier. It is known that the threshold $d/2$ is sharp in this case, so the condition $s> Q/2$ in Theorem \ref{thm:multipliers}\ref{en:multipliers_mh} cannot be weakened in general. One could also take arbitrary $\gamma_j \in (0,\infty)$, thus obtaining nonisotropic versions of the Mihlin--H\"ormander theorem on $\R^d$ (cf.\ \cite{kre_sur_1966,fabes_singular_1966}).
\item\label{en:ex1} Suppose that $L$ is a nonnegative selfadjoint operator on $L^2(X)$ satisfying, for some $h \in (0,\infty)$, the following heat kernel estimate: for all $a \in [0,\infty)$,
\begin{equation}\label{eq:heatkernel}
\sup_{t>0} \vvvert K_{e^{-tL}} \vvvert_{2,a,t^{1/h}} < \infty.
\end{equation}
Then Theorem \ref{thm:multipliers} can be applied with $n=1$, $U_1 = L$, $\gamma_1=h$ and $\Psi(\lambda) = e^{-\lambda}$. 
In the case $h>1$, the estimate \eqref{eq:heatkernel} can be obtained by H\"older's inequality from the following pointwise Gaussian-type heat kernel estimate: for all $t>0$ and $x,y \in X$,
\begin{equation}\label{eq:heatkernelpw}
|K_{e^{-t L}}(x,y)| \leq C \, \, V(y,t^{1/h})^{-1} \exp(-b (\dist(x,y)/t^{1/h})^{h/(h-1)}).
\end{equation}
This is one of the usual assumptions in abstract spectral multiplier theorems for a single operator $L$ (see, e.g., \cite{duong_plancherel-type_2002}).

\item\label{en:ex2} Let $L$ and $h$ be as in the previous example \ref{en:ex1}. Suppose moreover that $D$ is a selfadjoint operator on $L^2(X)$, that commutes strongly with $L$ and satisfies, for some $k \in (0,\infty)$, the following estimate: for all $a \in [0,\infty)$,
\begin{equation}\label{eq:heatkernel2}
\sup_{t>0} t^{k/h} \vvvert K_{D e^{-tL}} \vvvert_{2,a,t^{1/h}} < \infty.
\end{equation}
Then Theorem \ref{thm:multipliers} can be applied with $n=2$, $(U_1,U_2) = (L,D)$, $(\gamma_1,\gamma_2)=(h,k)$ and $\Psi(\lambda_1,\lambda_2) = (e^{-\lambda_1},\lambda_2 e^{-\lambda_1})$. As before, in the case $h>1$, the estimate \eqref{eq:heatkernel2} can be obtained from the following Gaussian-type estimate for the ``$D$-derivative'' of the heat kernel of $L$: for all $t>0$ and $x,y \in X$,
\begin{equation}\label{eq:dheatkernelpw}
|K_{D e^{-t L}}(x,y)| \leq C \, t^{-k/h} \, V(y,t^{1/h})^{-1} \exp(-b (\dist(x,y)/t^{1/h})^{h/(h-1)}).
\end{equation}
See, e.g., \cite{varopoulos_analysis_1992,auscher_positive_1994,dziubanski_notes_1994,ter_elst_heat_1997} for examples of differential operators $L,D$ satisfying \eqref{eq:heatkernelpw} and \eqref{eq:dheatkernelpw}; see also \cite{martini_multipliers_2010,martini_spectral_2011,martini_joint_2012} for examples of commuting operators.

\item Suppose that $X$ is the product $X_1 \times X_2$ of two doubling metric measure spaces of homogeneous dimensions $Q_1$ and $Q_2$, so $Q = Q_1 + Q_2$. For $j=1,2$ let $L_j$ be a nonnegative selfadjoint operator on $L^2(X_j)$ satisfying the analogue of \eqref{eq:heatkernel}: for some $h_j \in (0,\infty)$ and all $a \in [0,\infty)$,
\begin{equation}\label{eq:heatkernelj}
\sup_{t>0} \vvvert K_{e^{-tL_j}} \vvvert_{2,a,t^{1/h_j}} < \infty.
\end{equation}
Let $\tilde L_1 = L_1 \otimes I$ and $\tilde L_2 = I \otimes L_2$ be the corresponding operators on $L^2(X_1 \times X_2)$. Then Theorem \ref{thm:multipliers} can be applied with $n=2$, $(U_1,U_2) = (\tilde L_1,\tilde L_2)$, $(\gamma_1,\gamma_2) = (h_1,h_2)$, $\Psi(\lambda_1,\lambda_2) = (e^{-2\lambda_1-\lambda_2}, e^{-\lambda_1-2\lambda_2})$, since
\begin{align*}
K_{\Psi_1 \circ \epsilon_r(\tilde L_1,\tilde L_2)}(x,y) &= K_{e^{-2 r^{h_1} L_1}}(x_1,y_1) \, K_{e^{-r^{h_2} L_2}}(x_2,y_2),\\
K_{\Psi_2 \circ \epsilon_r(\tilde L_1,\tilde L_2)}(x,y) &=  K_{e^{-r^{h_1} L_1}}(x_1,y_1) \, K_{e^{-2 r^{h_2} L_2}}(x_2,y_2)
\end{align*}
and
\[
(1+\dist(x,y)/r)^a \leq (1+\dist_1(x_1,y_1)/r)^a (1+\dist_2(x_2,y_2)/r)^a.
\]
This gives an alternative proof and a generalization of the main result of \cite{sikora_multivariable_2009}, where Gaussian-type estimates like \eqref{eq:heatkernelpw} are required for each $L_j$, and only the case $h_1 = h_2$ is considered.
\end{enumerate}
\end{rem}

\begin{rem}
In the case the  map $\Psi$ in Theorem \ref{thm:multipliers} is smooth, one could replace the assumption \enref{en:multipliers_ass_inv} with the following:
\begin{enumerate}[label=(\ref*{en:multipliers_ass_inv}'),ref=\ref*{en:multipliers_ass_inv}']
\item\label{en:alt} $\Psi(0) \neq 0$, $d\Psi(0)$ is injective, $\Psi^{-1}(\Psi(0)) = \{0\}$ and $\Psi^{-1}(C)$ is compact for some compact neighborhood $C$ of $\Psi(0)$ in $\R^m$.
\end{enumerate}
Indeed, under the assumption \enref{en:alt}, the existence of the smooth local inverse $\Phi$ as in \enref{en:multipliers_ass_inv} can be obtained from the constant rank theorem.
On the other hand, smoothness of $\Psi$ is never used in the proof of Theorem \ref{thm:multipliers}. In fact, one could even weaken the smoothness assumption on the local inverse $\Phi$ of $\Psi$ as follows:
\begin{enumerate}[label=(\ref*{en:multipliers_ass_inv}''),ref=\ref*{en:multipliers_ass_inv}'']
\item\label{en:alt2} $\Psi(0) \neq 0$ and there exist an open neighborhood $\Omega$ of $\Psi(0)$ in $\R^m$ and a continuous map $\Phi : \Omega \to \R^n$ which is smooth on $\Omega \setminus \{\Psi(0)\}$ and such that $\Phi \circ \Psi$ is the identity on $\Psi^{-1}(\Omega)$.
\end{enumerate}
However, under this weaker assumption, one would obtain weaker versions of items \ref{en:multipliers_l2}, \ref{en:multipliers_schwartz}, \ref{en:multipliers_l1} of Theorem \ref{thm:multipliers}, where $F$ is constant in some neighborhood of $0$. A further generalization would be to consider several maps $\Psi$ instead of a single one and assume that, for all $a \in [0,\infty)$, the kernel estimate \eqref{eq:gaussiantype} is satisfied by some $\Psi$ that may depend on $a$. In the case $n=1$, this idea has been exploited in the study of the $L^p$ functional calculus for certain pseudodifferential operators $L$, where better and better estimates are available for the kernel of $e^{-tL^M}$ as $M \in \N$ grows \cite{dziubanski_remark_1989}. For the sake of clarity, we will not pursue this here.
\end{rem}

As mentioned above, one of the main ideas in the proof of Theorem \ref{thm:multipliers} is to use the map $\Psi$ as a ``change of variables'', in order to replace the (possibly unbounded) operators $U_1,\dots,U_n$ with the bounded operators $\Psi_j \circ \epsilon_r(U_1,\dots,U_n)$, $j=1,\dots,m$.
By means of this change of variables, the proof of Theorem \ref{thm:multipliers} is essentially reduced to the following result (which is a multivariate extension of a result of \cite{hebisch_functional_1995}).

For notational convenience, set $\exp_0(\lambda) = e^\lambda-1$.

\begin{prp}\label{prp:polyweightedboundedoperators}
Let $M_1,\dots,M_m$ be pairwise commuting, self-adjoint bounded operators on $L^2(X)$, admitting integral kernels $K_1,\dots,K_m$ respectively.
\begin{enumerate}[label=(\roman*)]
\item\label{en:bwl1e} Suppose that, for some $\kappa \geq 0$ and $a > 0$,
\begin{equation}\label{eq:polybounds}
\vvvert K_j\vvvert_{2,0} \leq \kappa, \qquad \vvvert K_j\vvvert_{1,a} \leq \kappa
\end{equation}
for $j=1,\dots,m$. Then, for all $h \in \Z^m$, the operator
\[
\exp_0(i (h_1 M_1 + \dots + h_m M_m))
\]
has an integral kernel $E_h$ satisfying, for all $b \in \left[0,a\right)$,
\[
\vvvert E_h\vvvert_{1,b} \leq C_{\kappa,a,b} \, |h|_1^{\gamma_{a,b}},
\]
where $|h|_1 = |h_1| + \dots + |h_m|$ and
the constants in the previous inequality depend only on the specified parameters and on the structure constants of $(X,\dist,\mu)$; in particular
\begin{equation}\label{eq:hebischgamma}
\gamma_{a,b} = 2^{\lfloor (b + Q/2)/(a-b) \rfloor} (1 + (b+Q/2)(1 + 1/(a-b))).
\end{equation}
\item\label{en:bwl2e} Suppose moreover that, for some $b \in [0,a)$,
\[
\vvvert K_j\vvvert_{2,b} \leq \kappa
\]
for $j=1,\dots,m$. Then, for all $h \in \Z^m$,
\begin{equation}\label{eq:l2polyestimate}
\vvvert E_h\vvvert_{2,b} \leq C_{\kappa,a,b} \, |h|_1^{\gamma_{a,b} + 1}.
\end{equation}
\item\label{en:bwl2fc} Under the previous assumptions, if $F \in L^2_s(\R^m)$ for some
\[
s > \gamma_{a,b} + 1 + m/2
\]
and $F(0) = 0$, then the operator $F(M_1,\dots,M_m)$ has an integral kernel satisfying
\[
\vvvert K_{F(M_1,\dots,M_m)} \vvvert_{2,b} \leq C_{m,\kappa,a,b,s} \|F\|_{L^2_s}.
\]
\end{enumerate}
\end{prp}

\begin{rem}
Proposition \ref{prp:polyweightedboundedoperators} is stated in terms of the norms $\vvvert \cdot \vvvert_{p,s}$, instead of the more general $\vvvert \cdot \vvvert_{p,s,r}$. However the norm $\vvvert \cdot \vvvert_{p,s,r}$ can be thought of as the norm $\vvvert \cdot \vvvert_{p,s}$ defined in terms of the rescaled distance $\dist/r$. Moreover it is easily seen that the ``rescaled space'' $(X,\dist/r,\mu)$ satisfies the same estimates \eqref{eq:doubling} and \eqref{eq:displacement} as the original space $(X,\dist,\mu)$, with the same structure constants. For this reason, it is not restrictive to omit the ``scale parameter'' $r$ in Proposition \ref{prp:polyweightedboundedoperators} and other statements where just a single scale is used.
\end{rem}

Before entering the proofs of the above statements, let us briefly recall a few basic facts about integral kernels. If an operator has integral kernel $K$, then its adjoint corresponds to the integral kernel $K^*$ given by
\[
K^*(x,y) = \overline{K(y,x)}.
\]
Composition of operators corresponds (under suitable integrability conditions on the kernels) to the following ``convolution'' of integral kernels:
\[
K_1 * K_2(x,y) = \int_X K_1(x,z) \, K_2(z,y) \,d\mu(z).
\]
The following lemma collects a few useful inequalities, among which is an extension of Young's inequality for convolution.

\begin{lem}\label{lem:convineq}
Let $H,K$ be the integral kernels of the operators $S,T$ respectively.
\begin{enumerate}[label=(\roman*)]
\item\label{en:convineq_hoelder}  For all $a,b \in[0,\infty)$ and $1 \leq p < q \leq \infty$,
\begin{equation}\label{eq:hoelderweight}
\vvvert K\vvvert_{p,a} \leq C_{p,q,a,b} \, \vvvert K\vvvert_{q,b} \qquad\text{if $b > a+ Q(1/p - 1/q)$.}
\end{equation}
\item\label{en:convineq_young} For all $a \in [0,\infty)$ and $1 \leq p,q,r \leq \infty$ with $1/p+1/q=1+1/r$,
\begin{equation}\label{eq:young}
\vvvert H * K\vvvert_{r,a} \leq (C'')^{1/p'} \, \vvvert H\vvvert_{p,a}^{p/r} \, \vvvert H^*\vvvert_{p,a+N/p'}^{p/q'} \, \vvvert K\vvvert_{q,a+N/p'}.
\end{equation}
The constants $C''$ and $N$ in \eqref{eq:young} are the same as in \eqref{eq:displacement}.
\item\label{en:convineq_comp} For all $p \in [1,\infty]$,
\begin{equation}\label{eq:composition}
\vvvert H * K \vvvert_{p,0} \leq \|S\|_{p \to p} \, \vvvert K\vvvert_{p,0}.
\end{equation}
\end{enumerate}
\end{lem}
\begin{proof}
\ref{en:convineq_hoelder}. This follows immediately from H\"older's inequality and \eqref{eq:intstandardweight}.

\ref{en:convineq_young}. Note that $1/r' = 1/p'+1/q'$. Let
\begin{align*}
\tilde H_1(x,y) &= H(x,y) \, V(y,1)^{1/p'} (1+\dist(x,y))^a,\\
\tilde H_2(x,y) &= H(x,y) \, V(x,1)^{1/p'} (1+\dist(x,y))^{a+N/p'},\\
\tilde K(y,z) &= K(y,z) \, V(z,1)^{1/q'} (1+\dist(y,z))^{a+N/p'}.
\end{align*}
Then, by \eqref{eq:displacement} and the inequality $(1+\dist(x,z)) \leq (1+\dist(x,y)) \, (1+\dist(y,z))$,
\[\begin{split}
&V(z,1)^{1/r'} |H(x,y) K(y,z)| (1+\dist(x,z))^{a} \\
&\qquad\leq (C'')^{1/p'} | \tilde H_1(x,y)|^{p/r} |\tilde H_2(x,y)|^{p/q'} |\tilde K(y,z)|^{q/r} |\tilde K(y,z)|^{q/p'}.
\end{split}\]
Integration in $y$ and H\"older's inequality yield
\[\begin{split}
V(z,1)^{1/r'} &|H*K(x,z)| (1+\dist(x,z))^{a} \\
&\leq (C'')^{1/p'} \left(\int_X |\tilde H_1(x,y)|^{p} |\tilde K(y,z)|^{q} \,d\mu(y) \right)^{1/r} \\ 
&\qquad\times\left(\int_X |\tilde H_2(x,y)|^{p} \,d\mu(y) \right)^{1/q'} \left( \int_X |\tilde K(y,z)|^{q} \,d\mu(y) \right)^{1/p'}.
\end{split}\]
By raising both sides to the power $r$, integrating in $x$ and taking the essential supremum in $z$, the inequality \eqref{eq:young} follows.

\ref{en:convineq_comp}. This follows easily from the fact that $H*K(\cdot,y) = S(K(\cdot,y))$.
\end{proof}


We are now ready to prove the above results. Let us see first how Theorem~\ref{thm:multipliers} is derived from Proposition~\ref{prp:polyweightedboundedoperators}.

\begin{proof}[Proof of Theorem~\ref{thm:multipliers}]
By replacing $\Psi$ with $\Psi \circ \epsilon_r$ for sufficiently small $r>0$ and possibly permuting the components of $\Psi$, we may assume the following:
\begin{itemize}
\item there exists $\psi \in C_c(\R^n)$ such that $\psi \cdot \Psi_1 = 1$ on $[-1,1]^n$;
\item there exist an open neighborhood $\Omega$ of $\Psi([-1,1]^n)$ in $\R^m$ and a smooth map $\Phi : \Omega \to \R^n$ such that $\Phi \circ \Psi$ is the identity on $\Psi^{-1}(\Omega) \cap \tilde\Sigma$.
\end{itemize}

\ref{en:multipliers_plancherel}. If $F : \R^n \to \C$ is supported in $[-1,1]^n$, then $F = (F \psi) \Psi_1$, hence
\[\begin{split}
\vvvert K_{F \circ \epsilon_r(U_1,\dots,U_n)}\vvvert_{2,0,r} &\leq \| (F \psi) \circ \epsilon_r(U_1,\dots,U_n) \|_{2 \to 2} \, \vvvert K_{\Psi_1 \circ \epsilon_r(U_1,\dots,U_n)} \vvvert_{2,0,r} \\
&\leq \|F\|_\infty \, \|\psi\|_\infty \, \vvvert K_{\Psi_1 \circ \epsilon_r(U_1,\dots,U_n)} \vvvert_{2,0,r}
\end{split}\]
for all $r>0$, and the conclusion follows by \eqref{eq:gaussiantype}.

\ref{en:multipliers_l2}.
We consider first the case $p=2$, i.e., we want to prove the inequality
\begin{equation}\label{eq:multipliers_l2}
\vvvert K_{F \circ \epsilon_r(U_1, \dots, U_n)} \vvvert_{2,b,r} \leq C_{s,b} \, \|F\|_{L^\infty_s}.
\end{equation}
for all $r > 0$, all $F : \R^n \to \C$ supported in $[-1,1]^n$, and all $s > b \geq 0$. 

Suppose first that $s > b + Q/2 + 2 + m/2$. Then we can find $a > 0$ sufficiently large so that $s > \gamma_{a,b} + 1 + m/2$,
where $\gamma_{a,b}$ is given by \eqref{eq:hebischgamma}.
By \eqref{eq:gaussiantype} and H\"older's inequality \eqref{eq:hoelderweight}, there exists $\kappa \in [0,\infty)$ such that
\[
\vvvert K_{\Psi_j \circ \epsilon_r(U_1,\dots,U_n)}\vvvert_{1,a,r} \leq \kappa, \qquad  \vvvert K_{\Psi_j \circ \epsilon_r(U_1,\dots,U_n)}\vvvert_{2,b,r} \leq \kappa
\]
for all $r > 0$ and $j = 1,\dots,m$. Take some $\eta \in C^\infty_c(\R^m)$ supported in $\Omega$ and identically $1$ on $\Psi([-1,1]^n)$. Then, for all $F : \R^n \to \C$ supported in $[-1,1]^n$,
\[
F = ((F \circ \Phi) \, \eta) \circ \Psi \qquad\text{on } \Sigma_\epsilon,
\]
so, for all $r > 0$,
\[
F \circ \epsilon_r(U_1,\dots,U_n) = ((F \circ \Phi) \, \eta) (\Psi \circ \epsilon_{r}(U_1,\dots,U_n)).
\]
By Proposition~\ref{prp:polyweightedboundedoperators}\ref{en:bwl2fc} applied to the rescaled space $(X,\dist/r,\mu)$ and the operators $\Psi_j \circ \epsilon_{r}(U_1,\dots,U_n)$, we then obtain that
\[
\vvvert K_{F \circ \epsilon_r(U_1,\dots,U_n)}\vvvert_{2,b,r} \leq C_{s,b} \|(F \circ \Phi) \, \eta\|_{L^2_s} \leq C_{s,b} \|F\|_{L^\infty_s},
\]
since $\eta$ is smooth and compactly supported and $\Phi$ is smooth on $\supp\eta$.

This gives \eqref{eq:multipliers_l2} for $b \geq 0$, $s > b + Q/2 + 2 + m/2$. On the other hand, \eqref{eq:multipliers_l2} also holds when $s=b=0$, by \ref{en:multipliers_plancherel}. The full range $s>b$ is then obtained by interpolation (cf., e.g., the proofs of \cite[Lemma 1.2]{mauceri_vectorvalued_1990}, \cite[Lemma 4.3]{duong_plancherel-type_2002}, \cite[Theorem 2.7]{martini_joint_2012}).

So the proof of the case $p=2$ is concluded. Note now that the case $p < 2$ follows from the case $p=2$ by H\"older's inequality \eqref{eq:hoelderweight}. As for $p>2$, take any real-valued $\xi \in C^\infty_c(\R^n)$ such that $\xi = 1$ on $[-1,1]^n$, and let $q$ be defined by $1/q=1/2+1/p$. Then $q<2$ and, from what we have just proved, it follows that
\begin{equation}\label{eq:cutoffestimate}
\sup_{r>0} \vvvert K_{\xi \circ \epsilon_r(U_1,\dots,U_n)} \vvvert_{q,a,r} < \infty
\end{equation}
for all $a \in [0,\infty)$. Moreover, for all $F : \R^n \to \C$ supported in $[-1,1]^n$, we have that $F = \xi \cdot F$; hence, by Young's inequality \eqref{eq:young} and the estimates \eqref{eq:multipliers_l2} and \eqref{eq:cutoffestimate},
\[\begin{split}
\vvvert K_{F \circ \epsilon_r(U_1,\dots,U_n)} \vvvert_{p,b,r} &\leq C_p \vvvert K_{\xi \circ \epsilon_r(U_1,\dots,U_n)} \, \vvvert_{q,b+N/q',r} \vvvert K_{F \circ \epsilon_r(U_1,\dots,U_n)} \vvvert_{2,b+N/q',r} \\
&\leq C_{p,b} \|F\|_{L^\infty_s}
\end{split}\]
for all $r \in (0,\infty)$ and $b,s \in [0,\infty)$ with $s > b+N/q' = b+N(1/2-1/p)$.

\ref{en:multipliers_schwartz}. Let $a \in [0,\infty)$ and $p\in[1,\infty]$, and set $s = a+Q(1/p-1/2)_+ + N(1/2-1/p)_+ +1$. Let $\phi_0 \in C_c(\R^n)$ and $\phi \in C_c(\R^n \setminus \{0\})$ be supported in $[-1,1]^n$ and such that $\phi_0 + \sum_{k > 0} \phi_k = 1$, where $\phi_k = \phi \circ \epsilon_{2^{-k}}$ for all $k > 0$. Then, by \ref{en:multipliers_l2}, for all $r>0$,
\[
\vvvert K_{(F \phi_0) \circ \epsilon_r(U_1,\dots,U_n)}\vvvert_{p,a,r} \leq C_{p,a} \| F \phi_0 \|_{L^\infty_{s}},
\]
and moreover, for all $k \in \N \setminus\{0\}$,
\[\begin{split}
\vvvert K_{(F \phi_k) \circ \epsilon_r(U_1,\dots,U_n)}\vvvert_{p,a,r} &\leq C_p \, 2^{kQ/p'} \, \vvvert K_{(F \phi_k) \circ \epsilon_r(U_1,\dots,U_n)}\vvvert_{p,a,2^{-k} r} \\
&\leq C_{p,a} \, 2^{kQ/p'} \, \| (F \circ \epsilon_{2^k}) \phi \|_{L^\infty_{s}},
\end{split}\]
by \eqref{eq:doubling}. Since it is easily seen that
\[
\| F \phi_0 \|_{L^\infty_{s}} + \sum_{k>0} 2^{kQ/p'} \, \| (F \circ \epsilon_{2^k}) \phi \|_{L^\infty_{s}}
\]
is controlled by some Schwartz norm of $F$, the conclusion follows.

\ref{en:multipliers_l1}. This follows immediately from \ref{en:multipliers_l2} and \eqref{eq:l1norm}.

\ref{en:multipliers_proj}. Note that $E(\{0\}) = \chr_{\{0\}}(U_1,\dots,U_n)$. Hence, by \ref{en:multipliers_plancherel},
\[
\vvvert K_{E(\{0\})}\vvvert_{2,0,r} \leq C
\]
uniformly in $r > 0$, so
\[
\int_X |K_{E(\{0\})}(x,y)|^2 \,d\mu(y) \leq C V(y,r)^{-1}
\]
for a.e.\ $y \in X$ and all $r>0$. If $\mu(X) = \infty$, then the right-hand side is infinitesimal as $r \to \infty$, hence $K_{E(\{0\})} = 0$ a.e.\ and consequently $E(\{0\}) = 0$.

If instead $\mu(X) < \infty$, then by \eqref{eq:doubling} we can find $r > 0$ such that $B(y,r) = X$ for all $y \in X$. Therefore, by H\"older's inequality,
\[
\vvvert K_{E(\{0\})}^*\vvvert_{1,0} = \vvvert K_{E(\{0\})}\vvvert_{1,0} \leq \vvvert K_{E(\{0\})}\vvvert_{2,0,r} \leq C,
\]
hence $\|E(0)\|_{p\to p} \leq C < \infty$ for $1 \leq p \leq \infty$ by \eqref{eq:l1norm} and interpolation.

\ref{en:multipliers_mh}. Note that $F = F(0) \chr_{\{0\}} + F \chr_{\R^n \setminus \{0\}}$, so $F(U_1,\dots,U_n) = F(0) E(\{0\}) + (F \chr_{\R^n \setminus \{0\}})(U_1,\dots,U_n)$. Hence, because of \ref{en:multipliers_proj}, it is not restrictive to assume in the following that $F(0)=0$. In particular
\[
\|F(U_1,\dots,U_n)\|_{2 \to 2} \leq \|F\|_\infty \leq  C_{\chi,s} \sup_{R>0} \|(F \circ \epsilon_R) \,\chi\|_{L^\infty_s}.
\]

Let $\omega \in C^\infty_c(\R^n \setminus \{0\})$ be such that $\supp \omega \subseteq [-1,1]^n$ and $\sum_{k \in \Z} \omega_k(\lambda) = 1$ for $\lambda \neq 0$, 
where $\omega_k = \omega \circ \epsilon_{2^{-k}}$. Take moreover a cutoff $\eta \in C^\infty_c(\R^n)$ which is identically $1$ on $[-1,1]^n$. Hence, by \ref{en:multipliers_l2}, the operators $A_t = \eta \circ \epsilon_t(U_1,\dots,U_n)$ satisfy
\[
\sup_{t>0} \vvvert K_{A_t} \vvvert_{\infty,b,t} < \infty
\]
for all $b \in [0,\infty)$, and in particular they satisfy the ``Poisson-type bounds'' of \cite[eq.\ (2) and (3)]{duong_singular_1999}.

Fix $a \in \left]Q/2,s\right[$ and an integer $M > s$. For $r,t > 0$, $k \in \Z$ and a.e.\ $y \in X$, we then have
\begin{multline*}
\int_{X \setminus B(y,r)} |K_{(F \omega_k)(U_1,\dots,U_n) \, (I-A_t)}(x,y)|\, d\mu(x) \\
\leq \left(\int_{X \setminus B(y,r)} \frac{(1 + 2^k \dist(x,y))^{-2a}}{V(y,2^{-k})} \, d\mu(x) \right)^{1/2} \, \vvvert K_{(F \omega_k (1-\eta \circ \epsilon_t))(U_1,\dots,U_n)} \vvvert_{2,a,2^{-k}} \\
\leq C_{\beta} (1+2^k r)^{Q/2 - a} \| (F \circ \epsilon_{2^k})\, \omega\|_{L^\infty_s}  \|(1 - \eta) \circ \epsilon_{2^k t} \|_{C^M(\supp \omega)}
\end{multline*}
by H\"older's inequality, \eqref{eq:doubling} and \ref{en:multipliers_l2}.

Note now that the quantity $\|(1 - \eta) \circ \epsilon_{t} \|_{C^M(\supp \omega)}$ is continuous in $t$, vanishes for $t \leq 1$ and is constant for $t$ large, therefore is bounded uniformly in $t > 0$. We then conclude that the series $\sum_{k \in \Z} K_{(F \omega_k (1-\eta \circ \epsilon_t))(U_1,\dots,U_n)}$ converges, away from the diagonal of $X \times X$, to a function $K_t$ satisfying
\[\begin{split}
\esssup_{y \in X} \int_{X \setminus B(y,r)} |K_t(x,y)|\, d\mu(x) &\leq C_s \sum_{k \tc 2^k > 1/t} (1+2^k r)^{Q/2-a} \|(F \circ \epsilon_{2^k}) \,\omega\|_{L^\infty_s} \\
&\leq C_{\chi,s} (r/t)^{Q/2-a} \sup_{R>0} \| (F \circ \epsilon_R) \, \chi\|_{L^\infty_s},
\end{split}\]
and it is easily checked that $K_t$ is the off-diagonal kernel of
\[(F (1 - \eta \circ \epsilon_t))(U_1,\dots,U_n) = F(U_1,\dots,U_n) \,(1-A_t).\]
If we take $r = t$ in the previous inequality, then \cite[Theorem 1]{duong_singular_1999} implies that $F(U_1,\dots,U_n)$ is of weak type $(1,1)$ and bounded on $L^p$ for $1 < p \leq 2$, with norm controlled by $\sup_{R>0} \| (F \circ \epsilon_R) \, \chi\|_{L^\infty_s}$. For $2 < p < \infty$, it is sufficient to apply the result just obtained to the function $\overline{F}$.
\end{proof}

We are left with the proof of Proposition~\ref{prp:polyweightedboundedoperators}.
The proof will require some preliminary considerations and lemmas.

By Young's inequality \eqref{eq:young}, for all $a \in [0,\infty)$, the norm $\vvvert \cdot\vvvert_{1,a}$ is submultiplicative, i.e.,
\[
\vvvert K_1 * K_2 \vvvert_{1,a} \leq \vvvert K_1 \vvvert_{1,a} \, \vvvert K_2 \vvvert_{1,a};
\]
consequently, the space of kernels
\[
B_a^0 = \{K  \tc \vvvert K\vvvert_{1,a}, \vvvert K^*\vvvert_{1,a} < \infty\},
\]
endowed with the norm
\begin{equation}\label{eq:starnorm}
\|K\|_{B_a} = \max\{ \vvvert K\vvvert_{1,a}, \vvvert K^*\vvvert_{1,a}\}
\end{equation}
and with the operations of convolution and involution, is a Banach $*$-algebra (with isometric involution). 

Let us denote by $T_K$ the operator corresponding to the integral kernel $K$. The identity \eqref{eq:l1norm}, together with interpolation, shows that the correspondence $K \mapsto T_K$ embeds $B_a^0$ continuously into the space $\Bdd(L^p)$ of bounded operators on $L^p(X)$ for all $p \in [1,\infty]$, with
\begin{equation}\label{eq:opnormstarnorm}
\|K\|_{p \to p} := \|T_K\|_{p \to p} \leq \|K\|_{B_a}.
\end{equation}

Since $B_a^0$ need not be unital, we formally introduce an identity element $I$, i.e., we consider the unitization $B_a = \C I \oplus B_a^0$, which is a Banach $*$-algebra with norm
\[
\|\lambda I + H \|_{B_a} = |\lambda| + \|H\|_{B_a}
\]
for every $\lambda \in \C$ and $H \in B_a^0$. Notice that, if we extend analogously the definition of the $\vvvert\cdot\vvvert_{1,1,a}$-norm by setting
\[
\vvvert\lambda I + H \vvvert_{1,a} = |\lambda| + \vvvert H\vvvert_{1,a},
\]
then the extension is still a submultiplicative norm, and \eqref{eq:starnorm} holds for every $K \in B_a$. Moreover the embedding into $\Bdd(L^p)$ extends to the whole $B_a$ (with possible loss of injectivity), together with the inequality \eqref{eq:opnormstarnorm}.

The formal introduction of an identity $I$ makes the manipulation of kernels easier. For instance, for every $K \in B_\phi$, the exponential $e^K$ is defined via power series as an element of $B_a$, and $\|e^K\|_{B_a} \leq e^{\|K\|_{B_a}}$; moreover the corresponding operator on $L^2(X)$ is nothing else than the exponential $e^{T_K}$. In the case $K$ is a ``genuine kernel'', i.e., $K \in B_a^0$, then $e^K \in I + B_a^0$ (as it is clear by inspection of the power series), therefore $e^K - I \in B_a^0$ is the kernel of the operator $\exp_0(T_K)$.

Let $M_1,\dots,M_m \in \Bdd(L^2)$ be pairwise commuting, self-adjoint operators admitting integral kernels $K_1,\dots,K_m$. According to the above discussion,
under the hypothesis \eqref{eq:polybounds}, for all $h \in \Z^m$, the operator $\exp_0(i (h_1 M_1 + \dots + h_m M_m))$ has integral kernel
\[
E_h = A_1^{* h_1} * \dots * A_m^{* h_m} - I;
\]
here $A_j = e^{i K_j}$ and $A_j^{* h_j}$ is the iterated convolution of $|h_j|$ factors of the form $A_j$ or $A_j^*$, depending on the sign of $h_j$.
Proposition~\ref{prp:polyweightedboundedoperators}\ref{en:bwl1e} will then be proved by showing that, for all $b \in [0,a)$ and $h \in \Z^m \setminus \{0\}$,
\begin{equation}\label{eq:convestimate}
\|A_1^{* h_1} * \dots * A_m^{* h_m}\|_{B_{b}} \leq C_{\kappa,a,b} |h|_1^{\gamma_{a,b}}.
\end{equation}
From now on, we will assume that $h \in \N^m$; the proof in the general case can be obtained by replacing some of the $A_j$ with $A_j^*$ in the argument below.

The idea is to decompose each $A_j$, or rather the kernel $A_j' = A_j - I$, into pieces supported at different distances from the diagonal. Namely, for some parameter $r \geq 1$ (which will be fixed later) we set
\[
A_{j,k}(x,y) = \begin{cases}
A'_j(x,y) &\text{if $e^{k} r \leq \dist(x,y) < e^{k+1} r$,} \\
0         &\text{otherwise,}
\end{cases}
\]
for $k > 0$, and $A_{j,0} = I + A'_{j,0}$, where
\[
A'_{j,0}(x,y) = \begin{cases}
A'_j(x,y) &\text{if $\dist(x,y) < er$,} \\
0         &\text{otherwise.}
\end{cases}
\]
By dominated convergence, $A_j' = A'_{j,0} + \sum_{k > 0} A_{j,k}$ in $B_{a}^0$, and consequently $A_j = \sum_{k \geq 0} A_{j,k}$ in $B_{a}$. By so decomposing each factor in $A_1^{* h_1} * \dots * A_m^{* h_m}$, we obtain an infinite sum of products of the form
\[P_{\alpha,\beta} = A_{\alpha_1,\beta_1} * \dots * A_{\alpha_n,\beta_n},\]
where $n = |h|_1$, $\alpha = (\alpha_1,\dots,\alpha_n) \in \{1,\dots,m\}^n$, $\beta = (\beta_1,\dots,\beta_n) \in \N^n$; for future convenience, we set $\II_n = \{1,\dots,m\}^n \times \N^n$, and
\[
|\beta|_H = \left| \{ u \tc \beta_u \neq 0 \}\right|.
\]

In order to estimate the $B_{b}$-norm of the products $P_{\alpha,\beta}$, we use the fact that $\delta = a - b > 0$; hence, for $k > 0$,
\begin{equation}\label{eq:offzeroestimate}
\|A_{j,k}\|_{2\to 2} \leq \|A_{j,k}\|_{B_{b}} \leq e^{-k\delta} r^{-\delta} \|A_j'\|_{B_{a}} \leq e^{-k\delta} r^{-\delta} \tilde\kappa,
\end{equation}
where $\tilde\kappa = e^\kappa$. For $k = 0$ this does not work, however we can exploit cancellation from the $L^2$-theory: since $\|A_j\|_{2\to 2} = \|e^{iT_{K_j}}\|_{2 \to 2} \leq 1$ by spectral theory, we also have
\begin{equation}\label{eq:onzeroestimate1}
\|A_{j,0}\|_{2 \to 2} \leq \|A_j\|_{2\to 2} + \|A_j - A_{j,0}\|_{B_{b}} \leq 1 + r^{-\delta} \tilde\kappa;
\end{equation}
moreover, by \eqref{eq:polybounds} and Young's inequality \eqref{eq:young},
\begin{equation}\label{eq:onzeroestimate2}
\vvvert A_{j,0}'\vvvert_{2,0} \leq \vvvert A_{j}'\vvvert_{2,0} \leq \sum_{n>0} \left\vvvert \frac{(i K_j)^{*n}}{n!} \right\vvvert_{2,0} \leq \sum_{n>0} \frac{\|K_j\|_{B_0}^{n-1} \vvvert K_j\vvvert_{2,0}}{n!} \leq \tilde\kappa
\end{equation}
and the same holds for $(A_{j,0}')^*$.

Set $\tilde b = b + Q/2$; we are then ready to prove

\begin{lem}
For all $n \in \N \setminus \{0\}$ and $(\alpha,\beta) \in \II_n$,
\begin{equation}\label{eq:rougherprodestimate}
\|P_{\alpha,\beta}\|_{B_{b}} \leq C_{\kappa,b} \, n \left(r \sum_{u=1}^n e^{\beta_u}\right)^{\tilde b} e^{-\delta |\beta|_1} (r^{-\delta} \tilde\kappa)^{|\beta|_H} (1+r^{-\delta} \tilde\kappa)^n.
\end{equation}
In particular, if $n \leq (r^{-\delta} \tilde\kappa)^{-1}$, then
\begin{equation}\label{eq:rawprodestimate}
\|P_{\alpha,\beta}\|_{B_{b}} \leq C_{\kappa,b} \, n^{1+\tilde b} \left(r e^{|\beta|_\infty}\right)^{\tilde b} e^{-\delta |\beta|_1} (r^{-\delta} \tilde\kappa)^{|\beta|_H},
\end{equation}
where $|\beta|_\infty = \max\{\beta_1,\dots,\beta_n\}$.
\end{lem}
\begin{proof}
Note that \eqref{eq:rawprodestimate} follows from \eqref{eq:rougherprodestimate} because $(1+r^{-\delta} \tilde\kappa)^n \leq (1+1/n)^n \leq e$ whenever $n \leq (r^{-\delta} \tilde\kappa)^{-1}$. We are then reduced to proving \eqref{eq:rougherprodestimate}.

By \eqref{eq:offzeroestimate} and \eqref{eq:onzeroestimate1}, a simpler estimate can be immediately obtained, for the $\Bdd(L^2)$-norm: for all $n \in \N$ and $(\alpha,\beta) \in \II_n$,
\[
\|P_{\alpha,\beta}\|_{2 \to 2} \leq e^{-\delta |\beta|_1} (r^{-\delta} \tilde\kappa)^{|\beta|_H} (1+r^{-\delta} \tilde\kappa)^n.
\]
From this, \eqref{eq:composition} and \eqref{eq:onzeroestimate2} we then get, for all $j\in\{1,\dots,m\}$,
\[
\vvvert P_{\alpha,\beta} * A'_{j,0} \vvvert_{2,0} \leq \tilde\kappa \, e^{-\delta |\beta|_1} (r^{-\delta} \tilde\kappa)^{|\beta|_H} (1+r^{-\delta} \tilde\kappa)^n.
\]
On the other hand, it is easily proved that $(P_{\alpha,\beta} * A'_{j,0})(x,y)$ vanishes for $d(x,y) \geq R$, where $R = e(r + r \sum_{u=1}^n e^{\beta_u}) \geq 1$. Hence, by H\"older's inequality and the doubling condition \eqref{eq:doubling},
\[
\vvvert P_{\alpha,\beta} * A'_{j,0}\vvvert_{1,b} \leq (1+R)^b \vvvert P_{\alpha,\beta} * A'_{j,0}\vvvert_{1,0} \leq C_b R^{\tilde b} \vvvert P_{\alpha,\beta}  *A'_{j,0}\vvvert_{2,0}
\]
and finally
\begin{equation}\label{eq:cptsptcontrol}
\vvvert P_{\alpha,\beta}  * A'_{j,0}\vvvert_{1,b} \leq C'_{\kappa,b} \left(r + r \sum_{u=1}^n e^{\beta_u}\right)^{\tilde b} e^{-\delta |\beta|_1} (r^{-\delta} \tilde\kappa)^{|\beta|_H} (1+r^{-\delta} \tilde\kappa)^n;
\end{equation}
here $C'_{\kappa,b}$ depends only on $\kappa$, $b$ and the constants in \eqref{eq:doubling}.

We now prove, for all  $n \in \N \setminus \{0\}$ and $(\alpha,\beta) \in \II_n$, the inequality
\[
\vvvert P_{\alpha,\beta}\vvvert_{1,b} \leq C_{\kappa,b} \, n \left(r \sum_{u=1}^n e^{\beta_u}\right)^{\tilde b} e^{-\delta |\beta|_1} (r^{-\delta} \tilde\kappa)^{|\beta|_H} (1+r^{-\delta} \tilde\kappa)^n,
\]
where $C_{\kappa,b} = C'_{\kappa,b} + 1$, by induction on $n$.

Set $\alpha' = (\alpha_1,\dots,\alpha_{n-1})$, $\beta' = (\beta_1,\dots,\beta_{n-1})$. If $\beta_n \neq 0$, then
\[
\vvvert P_{\alpha,\beta}\vvvert_{1,b} \leq \vvvert P_{\alpha',\beta'}\vvvert_{1,b} \, \vvvert A_{\alpha_n,\beta_n}\vvvert_{1,b};
\]
since the first factor $\vvvert P_{\alpha',\beta'}\vvvert_{1,b}$ is $1$ when $n=1$ and is controlled by the inductive hypothesis when $n>1$, while the second factor $\vvvert A_{\alpha_n,\beta_n}\vvvert_{1,b}$ is controlled by \eqref{eq:offzeroestimate}, the conclusion follows.
If instead $\beta_n = 0$, then
\[
\vvvert P_{\alpha,\beta}\vvvert_{1,b} \leq \vvvert P_{\alpha',\beta'}\vvvert_{1,b} + \vvvert P_{\alpha',\beta'} A'_{\alpha_n,0} \vvvert_{1,b};
\]
and the conclusion follows by majorizing $\vvvert P_{\alpha',\beta'}\vvvert_{1,b}$ as before and $\vvvert P_{\alpha',\beta'} A'_{\alpha_n,0} \vvvert_{1,b}$ 
by \eqref{eq:cptsptcontrol}.

An analogous argument proves the same estimate for $\vvvert P_{\alpha,\beta}^*\vvvert_{1,b}$ and the two estimates together give \eqref{eq:rawprodestimate}.
\end{proof}

From now on, the argument becomes purely Banach-algebraic. In fact, the inequality \eqref{eq:rawprodestimate} can be somehow improved via a combinatorial technique.

\begin{lem}\label{lem:combinatorial}
Let $\nu \in \N \setminus \{0\}$. For all $n \in \N \setminus \{0\}$ and $\beta \in \N^n$, there exists $I \subseteq \{1,\dots,n\}$ such that:
\begin{enumerate}[label=(\roman*)]
\item\label{en:combinatorial_a} $|I| \leq 2^{\nu-1} -1$;
\item\label{en:combinatorial_b} $\beta_j \neq 0$ for all $j \in I$;
\item\label{en:combinatorial_c} $\nu \sum_{j \in J} \beta_j \leq |\beta|_1$ for all $J \subseteq \{1,\dots,n\}$ with $|J| \leq |I| + 1$ and $J \cap I = \emptyset$.
\end{enumerate}
\end{lem}
\begin{proof}
Modulo reordering and padding with zeros, we may assume that the sequence $\beta_1,\dots,\beta_n$ is monotonic nonincreasing and $n \geq 2^{\nu}-1$. In particular
\[
|\beta|_1 \geq \sum_{k=0}^{\nu-1} \sum_{j=2^k}^{2^{k+1}-1} \beta_j,
\]
hence there exists $k \leq \nu-1$ such that
\[
|\beta|_1 \geq \nu \sum_{j=2^k}^{2^{k+1}-1} \beta_j.
\]
Set $I = \{j \in \{1,\dots,2^{k}-1\} \tc \beta_j \neq 0\}$. Clearly \ref{en:combinatorial_a} and \ref{en:combinatorial_b} are satisfied. Moreover, since $\beta$ is nonincreasing, if $I \neq \{1,\dots,2^{k}-1\}$, then, for every $J \subseteq \{1,\dots,n\}$ which is disjoint from $I$, we have $\beta_j = 0$ for all $j \in J$, so that \ref{en:combinatorial_c} is trivially satisfied. In the case $I = \{1,\dots,2^{k}-1\}$, instead, we have $J \subseteq \{2^k,\dots,n\}$, hence, if $|J| \leq |I| + 1 = 2^k$,
\[
\nu \sum_{j \in J} \beta_j \leq \nu \sum_{j = 2^k}^{2^{k+1}-1} \beta_j \leq |\beta|_1,
\]
because $\beta$ is nonincreasing, and we are done.
\end{proof}

\begin{lem}\label{lem:improvedpolyestimate}
Let $\nu \in \N \setminus \{0\}$. For all $n \in \N$ with $1 \leq n \leq (r^{-\delta} \tilde\kappa)^{-1}$, and for all $(\alpha,\beta) \in \II_n$,
\[
\|P_{\alpha,\beta}\|_{B_{b}} \leq C_{\kappa,b,\nu} \, r^{2^{\nu-1} \tilde b} n^{2^{\nu-1}(1+\tilde b)} e^{(\tilde b/\nu - \delta)|\beta|_1} (r^{-\delta} \tilde\kappa)^{|\beta|_H}.
\]
\end{lem}
\begin{proof}
Let $I \subseteq \{1,\dots,n\}$ be given by Lemma~\ref{lem:combinatorial} applied to $\nu$, $n$, $\beta$; in particular $l = |I| \leq 2^{\nu-1}-1$. Let $j_1,\dots,j_l$ be an increasing enumeration of the elements of $I$, so that $\beta_{j_1},\dots,\beta_{j_l}$ are nonzero, and set $j_0 = 0$, $j_{l+1} = n+1$. Then
\[
P_{\alpha,\beta} = P_{\alpha^{(0)},\beta^{(0)}} A_{\alpha_{j_1},\beta_{j_1}} P_{\alpha^{(1)},\beta^{(1)}} \dots A_{\alpha_{j_l},\beta_{j_l}} P_{\alpha^{(l)},\beta^{(l)}},
\]
where $\alpha^{(k)} = (\alpha_{j_{k-1} + 1},\dots,\alpha_{j_k - 1})$, $\beta^{(k)} = (\beta_{j_{k-1} + 1},\dots,\beta_{j_k - 1})$, and
\[
\nu \sum_{k=0}^{l} |\beta^{(k)}|_\infty \leq |\beta|_1.
\]
Moreover \eqref{eq:rawprodestimate} gives, for the $\beta^{(k)}$ which are nonempty,
\[
\|P_{\alpha^{(k)},\beta^{(k)}}\|_{B_{b}} \leq C_{\kappa,b} \, n^{1+\tilde b} r^{\tilde b} e^{\tilde b |\beta^{(k)}|_\infty} e^{-\delta |\beta^{(k)}|_1} (r^{-\delta} \tilde\kappa)^{|\beta^{(k)}|_H},
\]
whereas
\[
\|A_{\alpha_{j_k},\beta_{j_k}}\|_{B_{b}} \leq e^{-\beta_{j_k} \delta} r^{-\delta} \tilde\kappa
\]
by \eqref{eq:offzeroestimate}. Since
\begin{gather*}
|\beta|_1 = |\beta^{(0)}|_1 + \beta_{j_1} + |\beta^{(1)}|_1 + \dots + \beta_{j_l}  + |\beta^{(l)}|_1,\\
|\beta|_H = |\beta^{(0)}|_H + 1 + |\beta^{(1)}|_H + \dots + 1  + |\beta^{(l)}|_H,
\end{gather*}
the conclusion follows by multiplying these estimates together.
\end{proof}

\begin{proof}[Proof of Proposition~\ref{prp:polyweightedboundedoperators}]
\ref{en:bwl1e}.
From the previous discussion, we are reduced to proving \eqref{eq:convestimate} for $h \in \N^m \setminus \{0\}$. Having fixed such an $h$, we choose $r = (|h|_1 \tilde\kappa)^{1/\delta}$ and $\nu = \nu_{a,b} := \lfloor \tilde b/\delta \rfloor + 1$, so that $\epsilon_{a,b} := \delta - \tilde b /\nu > 0$. Hence Lemma~\ref{lem:improvedpolyestimate} gives, for $(\alpha,\beta) \in \II_{|h|_1}$,
\[\|P_{\alpha,\beta}\|_{B_{b}} \leq C_{\kappa,a,b} \, |h|_1^{\gamma_{a,b}} e^{-\epsilon_{a,b} |\beta|_1} |h|_1^{-|\beta|_H}.\]
Set now $\alpha = (1,\dots,1,2,\dots,2,\dots,m,\dots,m)$, with $h_1$ occurrences of $1$, $h_2$ occurrences of $2$, and so on. Then
\[\begin{split}
\|A_1^{* h_1} * \dots * A_m^{* h_m}\|_{B_{b}} &\leq \sum_{\beta \in \N^{|h|_1}} \|P_{\alpha,\beta}\|_{B_{b}} \\
&\leq C_{\kappa,a,b} \, |h|_1^{\gamma_{a,b}} \sum_{\beta \in \N^{|h|_1}} e^{-\epsilon_{a,b} |\beta|_1} |h|_1^{-|\beta|_H} \\
&= C_{\kappa,a,b} \, |h|_1^{\gamma_{a,b}} \left(1 + |h|_1^{-1} \sum_{k=1}^\infty e^{-\epsilon_{a,b} k}\right)^{|h|_1} \\
&\leq C_{\kappa,a,b} \, \exp\left(\sum_{k=1}^\infty e^{-\epsilon_{a,b} k}\right) \, |h|_1^{\gamma_{a,b}},
\end{split}\]
and we are done.

\ref{en:bwl2e}. We assume as before that $h_1,\dots,h_m \geq 0$. First of all, by arguing as in \eqref{eq:onzeroestimate2}, one immediately obtains that
\[
\vvvert A_j-I\vvvert_{2,b} \leq \tilde\kappa
\]
for $j=1,\dots,m$. Now we proceed inductively on $|h|_1$. The case $|h|_1 = 0$ is trivial. If instead $h_j \neq 0$ for some $j$, then we have
\[
E_h = E_{h'} + (A_1^{* h'_1} * \dots * A_m^{* h'_m}) * (A_j - I),
\]
where $h' = (h_1,\dots,h_{j-1},h_j - 1, h_{j+1}, \dots,h_m)$, hence
\[
\vvvert E_h\vvvert_{2,b} \leq C_{\kappa,a,b} (|h'|_1^{\gamma_{a,b} + 1} + |h'|_1^{\gamma_{a,b}}) \leq C_{\kappa,a,b} \, |h|_1^{\gamma_{a,b}+1}
\]
by \eqref{eq:convestimate}, Young's inequality \eqref{eq:young} and the inductive hypothesis.

\ref{en:bwl2fc}. We may assume, modulo rescaling, that $\kappa = 1$, so in particular the joint spectrum of $M_1,\dots,M_m$ is contained in $[-1,1]^m$. The Fourier series expansion
\[
F(\lambda) = \sum_{h \in \Z^m} \hat F(h) \, e^{i h \cdot \lambda} = \sum_{0 \neq h \in \Z^m} \hat F(h) \, (e^{i h \cdot \lambda} - 1)
\]
for $\lambda \in (-\pi,\pi)$ (where the last equality is due to the fact that $F(0) = 0$, and the convergence is uniform because $s > m/2$) then implies that
\[
F(M_1,\dots,M_m) = \sum_{0 \neq h \in \Z^m} \hat F(h) \, \exp_0(i (h_1 M_1 + \dots + h_m M_m))
\]
with convergence in $\Bdd(L^2)$. On the other hand,
\[
\sum_{0 \neq h \in \Z^m} |\hat F(h)| \, \vvvert E_h\vvvert_{2,b} \leq C_{a,b} \sum_{0 \neq h \in \Z^m} |\hat F(h)| \, |h|_1^{\gamma_{a,b} + 1} \leq C_{a,b,m,s} \|F\|_{L^2_s}
\]
by \eqref{eq:l2polyestimate} and H\"older's inequality, since $s > \gamma_{a,b} + 1 + m/2$, and we are done.
\end{proof}

\end{document}